\documentclass[reqno,11pt]{amsart}
\usepackage{newtxtext}
\usepackage {tikz}

\usetikzlibrary{cd,calc,arrows,positioning,angles,decorations.markings}
\tikzcdset{arrow style=tikz, diagrams={>=stealth}}

\tikzset{middlearrow/.style={
        decoration={markings,
            mark= at position 0.5 with {\arrow{#1}} ,
        },
        postaction={decorate}
    }
}
\tikzset{
Nedge/.append style={->,>=stealth', shorten <=6pt, shorten >=6pt},
every loop/.append style={min distance=12mm,shorten <=6pt, shorten >=6pt}
}

\newcommand{\Qbb}{\mathbb{Q}}
\newcommand{\Rbb}{\mathbb{R}}
\newcommand{\Zbb}{\mathbb{Z}}

\newcommand{\Acal}{\mathcal{A}}
\newcommand{\Fcal}{\mathcal{F}}
\newcommand{\Jcal}{\mathcal{J}}
\newcommand{\Scal}{\mathcal{S}}
\newcommand{\Pcal}{\mathcal{P}}
\newcommand{\Qcal}{\mathcal{Q}}
\newcommand{\Tcal}{\mathcal{T}}
\newcommand{\Gcal}{\mathcal{G}}
\newcommand{\Wcal}{\mathcal{W}}
\newcommand{\Proj}{\mathrm{Proj}}
\newcommand{\Aff}{\mathrm{Aff}}
\newcommand{\abf}{\mathbf{a}}
\newcommand{\bbf}{\mathbf{b}}
\newcommand{\Lbf}{\mathbf{L}}
\newcommand{\Nbf}{\mathbf{N}}
\newcommand{\Fbf}{\mathbf{F}}
\newcommand{\Ibf}{\mathbf{I}}
\newcommand{\Cbf}{\mathbf{C}}
\newcommand{\Ebf}{\mathbf{E}}
\newcommand{\Dbf}{\mathbf{D}}
\newcommand{\cbf}{\mathbf{c}}
\newcommand{\wbf}{\mathbf{w}}
\newcommand{\zbf}{\mathbf{z}}

\newcommand{\pp}{_{>0}}
\newcommand{\m}{^{-1}}
\newcommand{\bigcupdot}{\mathop{\dot\bigcup}}
\newcommand{\cupdot}{\mathbin{\dot\cup}}

\newcommand{\abs}[1]{\lvert#1\rvert}
\newcommand{\newword}[1]{\emph{#1}}
\newcommand{\floor}[1]{\lfloor #1 \rfloor}
\newcommand{\ceil}[1]{\lceil #1 \rceil}
\newcommand{\set}[1]{\{ #1 \}}
\newcommand{\ppmatrix}[4]{\bigl(\begin{smallmatrix}#1&#2\\#3&#4\end{smallmatrix}\bigr)}
\newcommand{\bbmatrix}[4]{\bigl[\begin{smallmatrix}#1&#2\\#3&#4\end{smallmatrix}\bigr]}
\newcommand{\transpose}{^\mathrm{T}}
\newcommand{\du}[1]{#1^\sharp}

\DeclareMathOperator{\PP}{P}
\DeclareMathOperator{\PSL}{PSL}
\DeclareMathOperator{\jump}{jump}

\theoremstyle{plain}
\newtheorem{theorem}{Theorem}[section]
\newtheorem{lemma}[theorem]{Lemma}
\newtheorem{corollary}[theorem]{Corollary}

\theoremstyle{definition}
\newtheorem{definition}[theorem]{Definition}
\newtheorem{remark}[theorem]{Remark}
\newtheorem{example}[theorem]{Example}

\numberwithin{equation}{section}

\begin{document}

\bibliographystyle{plain}

\sloppy

\title[Purely periodic continued fractions]{Purely periodic continued fractions\\
and graph-directed iterated function systems}

\author[G.~Panti]{Giovanni Panti}
\address{Department of Mathematics, Computer Science and Physics\\
University of Udine\\
via delle Scienze 206\\
33100 Udine, Italy}
\email{giovanni.panti@uniud.it}

\begin{abstract}
We describe Gauss-type maps as geometric realizations of certain  codes in the monoid of nonnegative matrices in the extended modular group. 
Each such code, together with an appropriate choice of unimodular intervals in $\PP^1\Rbb$, determines a dual pair of graph-directed iterated function systems, whose attractors contain intervals and constitute the domains of a dual pair of Gauss-type maps.
Our framework covers many continued fraction
algorithms (such as Farey fractions, Ceiling, Even and Odd, Nearest Integer, $\ldots$) and provides explicit dual algorithms and characterizations of those quadratic irrationals having a purely periodic expansion.
\end{abstract}

\thanks{\emph{2020 Math.~Subj.~Class.}: 11A55, 37E15}
\thanks{This is the author-created, un-copyedited final version of an article in \emph{The Ramanujan J.}, 2024. The official version is available Open Access at \texttt{https://link.springer.com/article/10.1007/s11139-024-00904-8}.}

\maketitle

\section{Introduction}\label{ref1}

In 1828 \'Evariste Galois, then a 17-year-old student at the Lyc\'ee Louis-le-Grand in Paris, published the following result~\cite{galois28}.

\begin{theorem}
Let\label{ref5} $\omega\in[0,1]$ be a quadratic irrational. Then $\omega$ has a purely periodic expansion as an ordinary continued fraction, $\omega=[0,\overline{a_1,\ldots,a_{p-1}}]$, if and only if its algebraic conjugate $\omega'$ is less than $-1$. If this happens, then we have $-1/\omega'=[0,\overline{a_{p-1},\ldots,a_1}]$.
\end{theorem}

In terms of dynamical systems, this amounts to a characterization of the set of numbers which have a purely periodic orbit under the Gauss map $F(x)=x\m-\floor{x\m}$, together with the fact that $F$ is selfdual up to a conjugation.
Replacing~$F$ with other maps piecewise-defined on appropriate intervals in $\PP^1\Rbb$ via matrices in the extended modular group $\PSL^\pm_2\Zbb$, we obtain more general Gauss-type maps. All of them have as eventually periodic points the set of quadratic irrationals in the domain of definition, but different maps on the same domain may have different sets of purely periodic points.

In this paper we extend the formalism introduced in~\cite{panti22} in two directions, firstly  by admitting finite unions of unimodular intervals (that is, images of $[0,\infty]$ by elements of the modular group) as domain of definition, and secondly by 
explicitly treating the accelerations of the resulting ``slow'' maps. In particular, we cover maps with not necessarily full branches, such as the ones related to Odd and Nearest Integer continued fractions, as well as those related to certain $\alpha$- and $(a,b)$-continued fractions (see~\cite{rieger79}, \cite{nakada81}, \cite{tanaka-ito81}, \cite{kraaikamp-schmidt-steiner12}, \cite{katok-ugarcovici12}, \cite{carminati-tiozzo13} and references therein). For simplicity's sake we restrict ourselves to maps with branches in the extended modular group; the extension to Hecke groups as treated in~\cite{panti22} is straightforward, but
burdens the notation and is anyhow pointless with an aim of characterizing purely periodic orbits, since even the problem of characterizing eventually periodic ones is, to the best of our knowledge, still wide open.
Referring to~\S\ref{ref2} for formal definitions, a simplified version of our main results, Theorems~\ref{ref19} and~\ref{ref31}, reads as follows.

\begin{theorem}
Let\label{ref3} $\Acal$ be an abstract continued fraction, and let be given finitely many unimodular intervals in $\PP^1\Rbb$ that constitute a geometric realization of $\Acal$. Then these data determine a Gauss-type multivalued map $F$ and a dual one $\du{F}$ on appropriate compact subsets $H$ and $K$ of $\PP^1\Rbb$. A quadratic irrational $\omega\in H$ has a purely periodic $F$-orbit if and only if its Galois conjugate $\omega'$ belongs to $K$; if this happens, then the $F$-orbit of $\omega$ and the $\du{F}$-orbit of $\omega'$ are ordinary single-valued ones, and correspond to each other via time reversal and conjugation.
This duality is preserved by passing to a jump acceleration $F_{\jump}$ of $F$, which then corresponds to the first-return map $\du{F}_R$ of $\du{F}$ on an appropriate subset $R$ of $K$.
\end{theorem} 

As simple examples we consider the continued fractions mentioned in the abstract, whose corresponding Gauss-type maps are as follows.
\begin{enumerate}
\item The Farey map~\cite{ito89} \cite[\S8]{isola11} on $[0,1]$ is given by
\[
x\mapsto\min(x(1-x)\m,(1-x)x\m).
\]
\item The Ceiling map~\cite[\S V]{zagier75} on $[0,1]$ is
\[
x\mapsto -x\m+\ceil{x\m}.
\]
\item The unfolded version of the Even map (see~\cite{schweiger82}, \cite{schweiger84}, \cite{dajani_et_al12} for the original versions of the Even and Odd maps)
on $[-1,1]$ is
\[
x\mapsto\abs{x\m}-(\text{the even integer nearest to $\abs{x\m}$}).
\]
\item The unfolded version of the Odd map on $[-1,1]$ is
\[
x\mapsto\abs{x\m}-(\text{the odd integer nearest to $\abs{x\m}$}).
\]
\item The Nearest Integer map~\cite{rieger79}, \cite[p.~399]{nakada81} on $[-1/2,1/2]$ is
\[
x\mapsto\abs{x\m}-(\text{the integer nearest to $\abs{x\m}$}).
\]
\end{enumerate}
Dropping for simplicity the statement about dual maps, and writing $\tau$ for the golden ratio $(\sqrt5+1)/2$, Theorem~\ref{ref3} yields the following characterizations
(analogous characterizations for the slow versions of (2)--(5) can be extracted from Examples~\ref{ref38}, \ref{ref23}, \ref{ref32}).

\begin{corollary}
Let\label{ref4} $\omega$ be a quadratic irrational in the domain of one of the maps (1)--(5). Then $\omega$ is purely periodic under the map if and only if its Galois conjugate $\omega'$ satisfies, respectively, the following inequalities.
\begin{enumerate}
\item $\omega'<0$.
\item $\omega'>1$.
\item $\omega'<-1$.
\item $\omega'\le-\tau-1$ in case $\omega<0$, or $\omega'\le-\tau+1$ in case $\omega>0$.
\item $\omega'\le-\tau-1$ in case $\omega<0$, or $\omega'\le-\tau$ in case $\omega>0$.
\end{enumerate}
\end{corollary}

Note that we cited Corollary~\ref{ref4} just as a convenient example covering familiar maps; as a matter of fact, those specific characterizations already appear in the literature~\cite{kraaikamp-lopes96}, \cite{boca-merriman18}, \cite{boca-siskaki21}. Our aim in the present paper is, on the one hand, in providing general statements (Theorems~\ref{ref19} and \ref{ref31}) covering a wide range of maps, and on the other in
developing a formalism whose ultimate goal is to clarify what a continued fraction actually \emph{is}. Granted that a final answer is probably impossible ---or maybe undesirable--- we however regard two issues as particularly relevant.
\begin{itemize}
\item[(a)] Clarify the combinatorial structure underlying a continued fraction algorithm.
\item[(b)] Clarify how a set of matrices obeying that combinatorial structure determines the domains of a dual pair of Gauss-type maps $F$, $\du{F}$ having that set of matrices as set of inverse branches.
\end{itemize}
In the simplest case of maps defined over the single unimodular interval $[0,1]$, these issues have been covered ---albeit not explicitly formulated--- in~\cite{panti22}, as follows.
\begin{itemize}
\item[(a${}'$)] The combinatorial structure is that of a maximal prefix code in the monoid of nonnegative matrices in the extended modular group.
\item[(b${}'$)] The domain of $F$ is $[0,1]$, while that of $\du{F}$ is the attractor of a certain ordinary Iterated Function System.
\end{itemize}
Essentially, the reader may think the above ``set of matrices'' as the set of inverse branches of her favorite slow Gauss-type map. Requiring that they form a ``code'' means that they generate a free monoid; this guarantees that their inverses unambiguously determine the map. Requiring maximality means that this monoid is not too thin, which ensures that the domains of both the map and of its dual contain intervals.

Unfortunately, the one-interval setting of~\cite{panti22} is too restrictive. For example, the leftmost branch of the Odd map is induced by the inverse of $B=\bbmatrix{}{-1}{1}{1}$, and there does not exist any proper interval in $\PP^1\Rbb$ which is mapped inside itself by $B$ (because otherwise successive iterates of $B$ would shrink the interval to a point, and thus~$B$ would have infinite order, instead of order~$3$). Also, the restriction to prefix codes makes the relation between a map and its dual asymmetric.

We recover the needed degree of flexibility by replacing ordinary IFS with graph-directed ones. This is a significant and unifying generalization,
that can be traced back to a series of papers (see~\cite{ghenciu_et_al17} and references therein) on the dimension spectrum of Gauss-type maps with countably many branches. For example, the matrix
$B$ above can now be seen as a contractive map from $[0,1]$ to a proper subinterval of $[-1,0]$.

We describe the needed combinatorial structure in Definition~\ref{ref6}. Again, the key requirement is that the labeled edges of the directing graph must constitute a maximal code.
Here the definitions are more involved, due to graph-directed structure; in particular, maximality has to be expressed in terms of the spectral radius of an appropriate incidence matrix. However, the consequences are the same, and we are able to show
that ---once appropriate unimodular intervals in $\PP^1\Rbb$, one for each node of the graph, have been determined--- the resulting IFS and its dual satisfy the Open Set Condition, so that the dual pair $F$, $\du{F}$ of Gauss-type maps of Theorem~\ref{ref3} can be unambiguously defined. Moreover, the attractors of these IFS, which are the natural domains of the maps, are guaranteed to contain intervals.

A characteristic of our approach, that initially appeared quite surprising to us, is that it reduces most of the proofs to combinatorial, rather than geometric, arguments. Indeed our main results, Theorems~\ref{ref19} and~\ref{ref31}, are ultimately based on the purely combinatorial Theorem~\ref{ref17}.

The structure of our paper is as follows: in~\S\ref{ref2} we define abstract continued fractions as codes over a certain structure~$\widetilde\Sigma$ that is a multi-node enlargement of the monoid $\Sigma$ of all nonnegative matrices in the extended modular group. We describe how choosing appropriate unimodular intervals in the real projective line converts an abstract continued fraction into a dual pair of graph-directed IFS, which in turn determine a dual pair of Gauss-type maps. In Remark~\ref{ref34} we briefly sketch how our approach compares with the more familiar natural extension construction, and we conclude the section by stating our main technical tool, Theorem~\ref{ref17}. In~\S\ref{ref13} we provide two detailed examples, 
the simple one of the Farey map and a much more involved one,
which is intended to display all delicate points of the construction. In~\S\ref{ref24} we prove Theorem~\ref{ref17}, and in~\S\ref{ref25} we state and prove the ``slow'' part of Theorem~\ref{ref3}, namely Theorem~\ref{ref19}. In~\S\ref{ref26} we treat the Schweiger jump operator, which had been previously encountered in~\S\ref{ref13}, and prove Theorem~\ref{ref22}, thus obtaining the ``fast'' part of Theorem~\ref{ref3}. We conclude the paper by applying our results to the maps~(2)--(5) of Corollary~\ref{ref4}.

\section{Gauss-type maps and graph-directed iterated function systems}\label{ref2}

Let $\Sigma$ be the monoid generated by the letters $l,n,f$, modulo the relations $fl=n\negthinspace f$, $fn=lf$, $f\negthinspace f=\epsilon(=\text{identity element})$. Every~~$\sigma\in\Sigma$ is uniquely expressible as a free word $w$ in $l$ and $n$, possibly followed by a single occurrence of $f$; we always assume the elements of $\Sigma$ are written in this normal form. We let $\ell(\sigma)$ be the word length of $w$, and define $\du{\sigma}$ to be the result of writing $\sigma$ backwards, exchanging $l$ with $n$, and using the relations to push the eventual $f$ to the end of the resulting word.
Clearly $\sharp$ is an involutory antiisomorphism that preserves $\ell$ (that is, $\epsilon^\sharp=\epsilon$ and, for every $\sigma,\tau\in\Sigma$, we have $\ell(\du{\sigma})=\ell(\sigma)$, $\sigma^{\sharp\sharp}=\sigma$ and $\du{(\sigma\tau)}=\du{\tau}\du{\sigma}$). Let $\set{0,\ldots,n-1}$ be a set of $n\ge1$ \newword{nodes}, and let $\widetilde\Sigma=
\set{i\sigma j:\sigma\in\Sigma\text{ and }i,j\in\set{0,\ldots,n-1}}$.
A partial binary operation is defined on $\widetilde\Sigma$ in the obvious way: the product $(i\sigma j)(h\tau k)$ is defined if and only if $j=h$ and, if so, has value $i\sigma\tau k$.
Both $\ell$ and $\sharp$ are naturally extended to $\widetilde\Sigma$ via
$\ell(i\sigma j)=\ell(\sigma)$ and $\du{(i\sigma j)}=j\du{\sigma}i$.

\begin{definition}
A \newword{code} over $\widetilde\Sigma$ is a subset $\Acal$ of $\widetilde\Sigma$ such that, whenever two products
$i_{t_1}\sigma_{t_1}j_{t_1}\cdots i_{t_r}\sigma_{t_r}j_{t_r}$
and 
$i_{q_1}\sigma_{q_1}j_{q_1}\cdots
i_{q_s}\sigma_{q_s}j_{q_s}$ of elements of $\Acal$ are defined and equal in $\widetilde\Sigma$, then
$r=s$ and $i_{t_k}\sigma_{t_k}j_{t_k}=i_{q_k}\sigma_{q_k}j_{q_k}$ for every~$k$.
\end{definition}

Clearly, if $\Acal$ is a code, then
no identity element $i\epsilon i$ belongs to $\Acal$,
and $\du{\Acal}$ is a code as well.

\begin{example}
The\label{ref10} simplest setting is with only one node $0$ ---which can then be dropped--- and no element of $\Acal$ containing the letter $f$.
In this case a code $\Acal$ is nothing else than a binary code~\cite[Chapter~6]{lothaire02}, \cite[Chapter~2]{berstel_et_al10}, with $\du{\Acal}$ the dual code. For example $\Acal=\set{l,ln,nn}$ is a code, because its dual $\du{\Acal}=\set{n,ln,ll}$ is a prefix code. The fact that $\Acal$ is a code means that every finite word over $\set{l,n}$ can be parenthesized in at most one way as a product of elements of $\Acal$; note that this does \emph{not} extend to infinite sequences. Indeed, the sequence $lnnn\ldots$ can be parenthesized in precisely two ways.
\end{example}

The monoid $\Sigma$ has two faithful geometric representations, a projective and an affine one.
The projective representation $\Proj$ is induced by mapping $l,n,f$, respectively, to the elements
\[
L=\begin{bmatrix}
1 & \\
1 & 1
\end{bmatrix},\quad
N=\begin{bmatrix}
1 & 1\\
 & 1
\end{bmatrix},\quad
F=\begin{bmatrix}
 & 1\\
1 & 
\end{bmatrix},
\]
of the \newword{extended modular group}, namely the group $\PSL^\pm_2\Zbb$ of all $2\times2$ matrices with integer entries and determinant either $1$ or $-1$, modulo the scalar subgroup~$\set{1,-1}$. We write elements of the extended modular group using  square brackets to emphasize that they are taken up to sign, blank entries denoting zeros. It is well known and easy to prove that $\Proj[\Sigma]$ is precisely the monoid of all nonnegative (again, up to multiplication by $-1$) matrices in $\PSL^\pm_2\Zbb$. Note that $\Proj(\du{\sigma})=\Proj(\sigma)\transpose$, the exponent denoting transpose.

The affine representation $\Aff$ of $\Sigma$ is induced by mapping $l,n,f$, respectively, to the maps on $\Rbb$ with dyadic coefficients
\[
\Lbf(x)=2\m x,\quad
\Nbf(x)=2\m x+2\m,\quad
\Fbf(x)=-x+1.
\]
It is readily seen that this is again a faithful representation, that $\Aff(\sigma)$ is a contraction of factor $2^{-\ell(\sigma)}$, and that all elements of $\Aff[\Sigma]$ map the real unit interval $[0,1]$ into itself.
The two representations $\Proj$ and $\Aff$ are topologically conjugate; this is proved in~\cite[Theorem~5.1]{panti22}, and used here in Theorem~\ref{ref28}.
 
There is a bijection between elements of the standard modular group $\PSL_2\Zbb$ (in which only matrices of determinant $1$ are allowed) and \newword{unimodular intervals} in~$\PP^1\Rbb$. Indeed, we identify the latter with the boundary of the Poincar\'e disk, and let $\bbmatrix{p'}{p}{q'}{q}$ correspond to the closed interval $[p/q,p'/q']$ described by going from $p/q$ to $p'/q'$ in the counterclockwise direction. Equivalently, $[p/q,p'/q']$ is the image of the base interval $[0,\infty]$ (which corresponds to the identity matrix) under $\bbmatrix{p'}{p}{q'}{q}$, with acts in the standard projective way $x\mapsto (p'x+p)/(q'x+q)$; as usual, we identify matrices with the maps they induce. Letting $S=\bbmatrix{}{-1}{1}{}$, we have that $\bbmatrix{p'}{p}{q'}{q}S$ corresponds to the \newword{complementary interval} $[p'/q',p/q]$ of 
$[p/q,p'/q']$.

Let $\Acal\subset\widetilde\Sigma$ be a finite code. We write $\Gcal_\Acal$ for the finite directed graph whose vertices are the nodes $0,\ldots,n-1$, and whose set of edges from $i$ to $j$ is $E_{ij}=\set{\text{elements of $\Acal$ of the form $i\sigma j$}}$.
We let $E_{ij}^{<\omega}$ be the set of all paths from $i$ to $j$, $E_i^\omega$ the set of all infinite paths from $i$, and $E^\omega=\bigcup_i E_i^\omega$. We also let $G_\Acal$ be the $n\times n$ matrix whose $ij$-th entry is $\sum\set{2^{-\ell(a)}:a\in E_{ij}}$.

\begin{definition}
An\label{ref6} \newword{abstract continued fraction} is a finite code $\Acal\subset\widetilde\Sigma$ such that $\Gcal_\Acal$ is strongly connected (that is, every set $E_{ij}^{<\omega}$ is nonempty) and $G_\Acal$ has spectral radius~$1$.
\end{definition}

Definition~\ref{ref6} makes precise the ``combinatorial structure'' alluded to in item~(a) in the introduction. The requirement of being a code ensures that the Gauss-type map to be introduced in Definition~\ref{ref16} is well-defined, while that on the spectral radius ensures (see the proof of Theorem~\ref{ref28}) that the domain of $\Fcal$ contains intervals.

By construction, the graph $\Gcal_{\du{\Acal}}$ has the same vertices as $\Gcal_\Acal$, and all arrows are reversed, the set of edges from $i$ to $j$ being $(\du{E})_{ij}=\set{\du{a}:a\in E_{ji}}$. Since~$\ell(a^\sharp)=\ell(a)$, we have $G_\du{\Acal}=G_\Acal\transpose$; thus, if $\Acal$ is an abstract continued fraction, so is its \newword{dual}~$\du{\Acal}$. 

Let us now fix $n$ unimodular intervals $I_0,\ldots, I_{n-1}$, one for each node; we abuse notation by writing $I_i$
both for the $i$-th interval and for the 
matrix in $\PSL_2\Zbb$ corresponding to it.

\begin{lemma}
Let\label{ref11} $i\sigma j\in\widetilde\Sigma$. Then the matrix $B_{i\sigma j}=I_i\Proj(\sigma)I_j\m$ maps $I_j$ to $I_i$, while 
its inverse $B_{i\sigma j}\m$ equals $(I_jS)\Proj(\du{\sigma})(I_iS)\m$ and maps $I_iS$ to $I_jS$. If $\ell(\sigma)\ge1$, then the images of these maps are proper subintervals of the target intervals.
\end{lemma}
\begin{proof}
It is clear that $\Proj(\sigma)$ maps $[0,\infty]$ into itself, and never onto itself unless $\sigma$ equals $\epsilon$ or $f$. Since these are the only elements $\sigma\in\Sigma$ such that $\ell(\sigma)=0$, the statements about $B_{i\sigma j}$ are immediate. As
\[
(I_jS)\Proj(\du{\sigma})(I_iS)\m=
I_jS\Proj(\sigma)\transpose S\m I_i\m=
I_j\Proj(\sigma)\m I_i\m=B_{i\sigma j}\m,
\]
and $\sharp$ preserves~$\ell$, the statements about $B_{i\sigma j}\m$ follow.
\end{proof}

Given any finite $\Acal\subset\widetilde\Sigma$ and intervals $I_i$ as above, the pair $(\Gcal_\Acal,\set{B_a}_{a\in\Acal})$ constitutes a \newword{graph-directed iterated function system} (g-d IFS).
This simply means that,
for each pair $I_i$, $I_j$, we are considering 
the finitely many projective maps from $I_j$ to $I_i$ determined
by $\set{B_a:a\in E_{ij}}$; see~\cite{mauldin-williams88}, \cite{edgar-golds99}, \cite{das-edgar05} for a detailed treatment.
As customary, if $a_0\ldots a_{t-1}$ is a finite path in $\Gcal_\Acal$ we abbreviate the product $B_{a_0}\cdots B_{a_{t-1}}$ by $B_{a_0\ldots a_{t-1}}$.
Note that, by definition,
the intervals $I_0,\ldots,I_{n-1}$ must be realized in pairwise disjoint copies of $\PP^1\Rbb$; we stress this fact by using disjoint union notations such as $\bigcupdot_iI_i=\set{(i,x):x\in I_i}$.

The following is a graph-directed version of the classical Ping-Pong Lemma for monoids~\cite[VII.A.2]{delaharpe00}.

\begin{lemma}
Let\label{ref14} $(\Gcal_\Acal,\set{B_a})$ be as above, and assume that:
\begin{itemize}
\item[(a)] There exist nonempty sets $U_i\subseteq I_i$ such that, for every $i$, the images $B_{i\sigma j}[U_j]$ are pairwise disjoint subsets of $U_i$.
\item[(b)] For every $j$ and every $a_0\ldots a_{p-1}\in E_{jj}^{<\omega}$ the map $B_{a_0\ldots a_{p-1}}$ is not the identity on $I_j$ (this is surely true if $\ell(a)\ge1$ for every $a\in\Acal$).
\end{itemize}
Then $\Acal$ is a code.
\end{lemma}
\begin{proof}
Let $a_0\ldots a_{p-1},b_0\ldots b_{q-1}\in E_{ij}^{<\omega}$ be different paths in $\Gcal_\Acal$. We want to show that the relative products are different elements of $\widetilde\Sigma$, and this follows from showing that $B_{a_0\ldots a_{p-1}}$ and $B_{b_0\ldots b_{q-1}}$ are different as functions from $I_j$ to $I_i$. By cancelling an eventual common prefix, we assume without loss of generality that either $a_0\not=b_0$, or $b_0\ldots b_{q-1}$ is the empty path and $a_0\ldots a_{p-1}$ a nonempty path, both from $j$ to $j$.
In the first case the conclusion follows from $B_{a_0\ldots a_{p-1}}[U_j]\cap B_{b_0\ldots b_{q-1}}[U_j]=\emptyset$, and in the second from condition~(b).
\end{proof}

If $\Acal$ is an abstract continued fraction, then (b) in Lemma~\ref{ref14} is obviously true because $\Acal$ is a code, while (a) is true in the stronger form that the sets $U_i$ can be taken to be open; this is proved in Theorem~\ref{ref28}. A graph-directed IFS that obeys (a) in this stronger form is said to satisfy the \newword{Open Set Condition}~\cite[Definition~3.11]{edgar-golds99}, \cite[\S3]{das-edgar05}. Note that ``open'' here, and all topological concepts in the paper, refer to the ambient topology of $\PP^1\Rbb$.

\begin{lemma}
Let\label{ref27} $\Acal$ be a code and let $\abf=a_0a_1\ldots\in E_{i_0}^\omega$, with $a_t=i_t\sigma_t i_{t+1}$. Then the intersection of the descending chain of unimodular intervals
\begin{equation}\label{eq7}
I_{i_0}\supseteq B_{a_0}I_{i_1}\supseteq
B_{a_0a_1}I_{i_2}\supseteq\cdots
\end{equation}
is a singleton. Letting $\pi(\abf)\in I_{i_0}$ be the element of that singleton, we have a well defined map $\pi$ from $E^\omega$ to $\bigcupdot_i I_i$, which is continuous with respect to the natural topology of $E^\omega$.
\end{lemma}
\begin{proof}
Let $0\le t(0)<t(1)<t(2)<\cdots$ be a sequence
of indices such that the intervals $I_{i_{t(k)}}$ are all equal, say to $I_j$. We claim that each interval of the chain
\begin{equation}\label{eq8}
I_j\supset B_{a_{t(0)}\cdots a_{t(1)-1}}I_j\supset
B_{a_{t(0)}\cdots a_{t(2)-1}}I_j\supset\cdots
\end{equation}
is a proper subinterval of the preceding one.
It suffices to show this fact for the first inclusion. Now, if that inclusion were an equality, then we would have that $a_{t(0)}\cdots a_{t(1)-1}$ equals either $j\epsilon j$ or $jfj$ in $\widetilde\Sigma$. But then
$a_{t(0)}\cdots a_{t(1)-1}$ would be equal either to its square or to its cube, which is impossible since $\Acal$ is a code.
By~\cite[Observation~3]{panti09} the intersection of the chain~\eqref{eq8} is a singleton, and therefore so is the intersection of~\eqref{eq7}. The statement about the continuity of $\pi$ is clear; see~\cite[\S2]{wang97} or \cite[p.~433]{edgar-golds99} for details.
\end{proof}

Letting $H_i=\pi[E_i^\omega]\not=\emptyset$, we have that
the \newword{attractor} $\bigcupdot_iH_i$ of the graph-directed IFS $(\Gcal_\Acal,\set{B_a})$ is compact.

\begin{remark}
The\label{ref21} attractor can equivalently be defined as the only fixed point of the graph-directed version of the Hutchinson operator $(T_0,\ldots,T_{n-1})\mapsto(T_0',\ldots,T_{n-1}')$, where each $T_i$ is a compact subset of $I_i$ and
\[
T_i'=\bigcup\bigl\{B_a[T_j]:a\in E_{ij}\bigr\}.
\]
Since the operator is a contraction with respect to the Hausdorff metric, starting from any $n$-tuple and repeatedly applying it yields a sequence of $n$-tuples $(T_0^t,\ldots,T_{n-1}^t)$ that converges to $(H_0,\ldots,H_{n-1})$.
It is usually ---but not always, even in the one-interval case, see~\cite[Example~4.5]{panti22}--- easy to guess what the limit may be, and then explicitly check that the guess is correct; Lemma~\ref{ref35} provides a nontrivial example.

For the convenience of the interested reader, we provide an easy implementation of the Hutchinson operator in SageMath, using the built-in construct \texttt{RealSet}.
Without loss of generality, we start with a list
\begin{verbatim}
T0=[RealSet([0,1]) for i in range(n)]
\end{verbatim}
of $n$ copies of $[0,1]$, and we assume given a list \texttt{B} of length $n$ of lists of length $n$, whose elements are (possibly empty) lists of $2\times2$ matrices, each of them mapping $[0,1]$ into itself.
Thus, \texttt{B[i][j]} is the list of matrices corresponding to the elements of $E_{ij}$, and all these matrices map the copy of $[0,1]$ indexed by $j$ to the copy of $[0,1]$ indexed by $i$. An implementation of the Hutchinson operator is then
\begin{verbatim}
def mq(m,q):
    q_vertices=[q.lower(), q.upper()] 
    return [(m[0,0]*x+m[0,1])/(m[1,0]*x+m[1,1]) \
            for x in q_vertices]
def HO(RS):
    RS0=[]
    for i in range(n):
        RS0i=RealSet()
        for j in range(n):
            for qq in [mq(m,q) \
                       for m in B[i][j] for q in RS[j]]:
                RS0i=RS0i.union(RealSet(qq))
        RS0=RS0+[RS0i]
    return RS0
\end{verbatim}
The code is straightforward. The function \texttt{mq} takes as input a matrix and an element of a \texttt{RealSet}, i.e., an interval, and returns the list of endpoints of the image interval. The function \texttt{HO} takes as input a list $(T_0,\ldots,T_{n-1})$ of \texttt{RealSets}, all of them subsets of $[0,1]$, and returns the image list $(T'_0,\ldots,T'_{n-1})$. Starting from \texttt{T0} and repeatedly applying \texttt{HO} one obtains the sequence $(T_0^t,\ldots,T_{n-1}^t)$.
\end{remark}

\begin{theorem}
Let\label{ref28} $\Acal$ be an abstract continued fraction and let $\set{I_i}_{i<n}$ be unimodular intervals. Then the graph-directed IFS $(\Gcal_\Acal,\set{B_a})$ satisfies the Open Set Condition, and all sets $H_0,\ldots,H_{n-1}$ in the attractor are regular (i.e., each one is the closure of its interior). In particular each $H_i$ contains intervals, and the Open Set Condition is satisfied by taking as open sets the interiors of the various $H_i$.
\end{theorem}
\begin{proof}
We may safely assume that each $I_i$ is a copy $[0,1]_i$ of the real unit interval; then, for every $i\sigma j\in\Acal$, we have $B_{i\sigma j}=L\Proj(\sigma)L\m$. We introduce a twin g-d affine IFS 
$(\Gcal_\Acal,\set{\Cbf_a})$ over $\set{[0,1]_i}$ by letting $\Cbf_{i\sigma j}=\Aff(\sigma)$.

An important fact about continued fractions is the existence of a homeomorphism of $[0,1]$, named the Minkowski Question Mark function~\cite{denjoy38}, \cite{salem43}, \cite{jordansahlsten16}, that conjugates the Farey map with the tent map $x\mapsto \min(2x,-2x+2)$.
For each~$i$, let $M_i:[0,1]_i\to[0,1]_i$ be a copy of the Minkowski homeomorphism. Then, by~\cite[Theorem~5.1]{panti22}, for every $i\sigma j\in\Acal$ the square
\[
\begin{tikzcd}[column sep=1.5em]
{[0,1]}_i \ar[d,"M_i"'] & &
{[0,1]}_j \ar[ll,"B_{i\sigma j}"']   \ar[d,"M_j"] \\
{[0,1]}_i & & {[0,1]}_j \ar[ll,"\Cbf_{i\sigma j}"]
\end{tikzcd}
\]
commutes. This means that the projective and the affine g-d IFS are topologically conjugate by the family $\set{M_i}$, and therefore the attractor $\bigcupdot_iH_i$ of the projective g-d IFS and the attractor $\bigcupdot_iH'_i$ of the affine one are related by $H'_i=M_i[H_i]$ for every~$i$. Since each $M_i$ is a homeomorphism, it suffices to show that each $H'_i$ is regular.

For each node $i$, let $\Ebf_{ii}^{<\omega}$ be the set of all affine maps of the form $\Cbf_{a_0\ldots a_{p-1}}$, where $a_0\ldots a_{p-1}\in E_{ii}^{<\omega}$. By the proof of Lemma~\ref{ref27}, the contraction rate of every element of $\Ebf_{ii}^{<\omega}$ is bounded above by $2\m$; by~\cite[pp.~435-436]{edgar-golds99}, the affine g-d IFS is contractive. Now, the claim in~\cite[p.~69]{panti22} shows that the identity function is an isolated point in each set $\set{\Cbf\m\Dbf:\Cbf,\Dbf\in\Ebf_{ii}^{<\omega}}$; therefore Condition~(3a) in~\cite[p.~140]{das-edgar05} holds, and the affine g-d IFS satisfies the Weak Separation Property. Noting that being a code amounts, in the terminology of~\cite{das-edgar05}, to distinguishing paths, \cite[Proposition~3.1]{das-edgar05} applies and therefore $(\Gcal_\Acal,\set{\Cbf_a})$ satisfies the Open Set Condition.

We show that each $H'_i$ is regular by adapting the proof of~\cite[Corollary~2.3]{schief94} to the graph-directed setting (to the best of our knowledge the possibility of this direct adaptation has not been noted in the literature). Let $(U_0,\ldots,U_{n-1})$ be a list of open sets $U_i\subseteq[0,1]_i$ satisfying the requirements of Lemma~\ref{ref14}(a). Letting $\lambda$ denote Lebesgue measure and $u_i=\lambda(U_i)>0$, we have for every $i$
\begin{align*}
\lambda\bigl(\bigcup\set{\Cbf_a[U_j]:a\in E_{ij}}\bigr)
=&\sum\set{\lambda(\Cbf_a[U_j]):a\in E_{ij}}\\
=&\sum\set{2^{-\ell(a)}u_j:a\in E_{ij}}\\
=&\sum_j g_{ij}u_j\\
\le&u_i;
\end{align*}
here $g_{ij}$ is the $ij$-th entry of $G_\Acal$.
Therefore, the vector inequality $G_\Acal(u_0\cdots u_{n-1})\transpose\le(u_0\cdots u_{n-1})\transpose$ holds componentwise. As a matter of fact we must have vector equality, because if we had strict inequality in a component the scalar product with the left Perron vector $v$ of $G_\Acal$ 
---which is strictly positive by the Perron-Frobenius theory---
would give $vu\transpose=vG_\Acal u\transpose<vu\transpose$, which is impossible.

Thus, the nested open sets
\[
\bigcup\set{\Cbf_a[U_j]:a\in E_{ij}}\subseteq U_i
\]
have the same Lebesgue measure, and this implies that the second is contained in the closure of the first (because the intersection of the second with the complement of the closure of the first is an open set of Lebesgue measure $0$, and thus is empty).
Therefore
\[
\overline{U_i}\subseteq
\bigcup\set{\overline{\Cbf_a[U_j]}:a\in E_{ij}}=
\bigcup\set{\Cbf_a[\overline{U_j}]:a\in E_{ij}},
\]
and since the reverse inclusion is clear, we have indeed equality. By the uniqueness of the attractor, $H'_i$ equals $\overline{U_i}$ for every $i$, and thus is a regular closed set.

In order to establish our last statement, we reset each $U_i$ to be the interior of the corresponding $H'_i$ and prove that, for each pair $a\in E_{ij}$, $b\in E_{ik}$ with $a\not=b$, we have:
\begin{enumerate}
\item $\lambda\bigl(\Cbf_a[H'_j]\cap\Cbf_b[H'_k]\bigr)=0$.
\item $\Cbf_a[U_j]\cup\Cbf_b[U_k]\subseteq U_i$.
\item $\Cbf_a[U_j]\cap\Cbf_b[U_k]=\emptyset$.
\end{enumerate}
Let $h_i=\lambda(H'_i)>0$; since $H'_i=\bigcup\set{\Cbf_a[H'_j]:a\in E_{ij}}$, we have $(h_0\cdots h_{n-1})\transpose\le G_\Acal(h_0\cdots h_{n-1})\transpose$ componentwise. Arguing as above, we have indeed vector equality; this implies (1) since otherwise we would have strict inequality in at least a component. Statement (2) is clear. If (3) were false, then $\Cbf_a[H'_j]\cap\Cbf_b[H'_k]$ would contain a nonempty open set, and (1) would also be false.
\end{proof}

Let $\Acal$ be an abstract continued fraction, and choose unimodular intervals~$\set{I_i}$.
Beyond the already introduced d-g IFS $(\Gcal_\Acal,\set{B_a})$ 
on $\set{I_i}$, we 
consider the dual g-d IFS $(\Gcal_{\du{\Acal}},\set{D_a})$ on the set of complementary intervals $\set{I_iS}$. Here, given $i\sigma j\in\Acal$, the function $D_{i\sigma j}:I_iS\to I_jS$ is defined by 
$D_{i\sigma j}=B_{i\sigma j}\m=(I_jS)\Proj(\du{\sigma})(I_iS)\m$ and, unless $\sigma$ equals $\epsilon$ or $f$, maps the domain interval to a proper subinterval of the target.

\begin{remark}
We\label{ref30} are labeling the maps $D_a$ using the edges  of $\Gcal_\Acal$, rather than those of $\Gcal_{\du{\Acal}}$; this leads to the annoying reverse of direction in the indices of the domain and target intervals. However the price is worth paying, because we can now use the same alphabet $a,b,a_0,a_1,\ldots\in\Acal$ for both IFS; note that $B_{a_0\ldots a_{p-1}}\m=D_{a_{p-1}\ldots a_0}$.
We define the letter $a$ to be \newword{parabolic} if it is of the form $a=il^ki$ or $a=in^ki$ for some node $i$ and exponent $k\ge1$. The name is justified because $B_a$ and $D_a$ are then parabolic matrices, that is, fix precisely one point of $\PP^1\Rbb$.
\end{remark}

According to the above conventions, an infinite path $\abf=a_0a_1\ldots\in(\du{E})^\omega$ is a path in $\Gcal_\Acal$ that follows reverse edges. Thus each $a_t$ has the form $i_{t+1}\sigma_t i_t$, and $\pi(\abf)$ is the intersection of the chain
\begin{equation*}
I_{i_0}S\supseteq D_{a_0}I_{i_1}S\supseteq
D_{a_0 a_1}I_{i_2}S\supseteq\cdots.
\end{equation*}
We let $\bigcupdot_iK_i$ be the attractor of the dual IFS, with $K_i=\pi[({\du{E}_i})^\omega]\subseteq I_iS$; all of our previous discussion, and notably Theorem~\ref{ref28}, applies. 

\begin{definition}
Let\label{ref16} $\Acal$ be an abstract continued fraction and choose unimodular intervals $\set{I_i}$.
The attractor $\bigcupdot_iH_i$ of the g-d IFS $(\Gcal_\Acal,\set{B_a})$ 
is both the domain and the image of a multivalued map $\Fcal$ that is conjugate to the shift $S$ on $E^\omega$ via~$\pi$: explicitly, if $x=\pi(\abf)$ with $a_0=i_0\sigma_0i_1$, then we set $\Fcal(i_0,x)=(i_1,\pi(S\abf))$, and say that $B_{a_0}\m$ \newword{acts} on $x$ (indeed, $\pi(S\abf)=B_{a_0}\m(x)$).
Such maps are sometimes called cookie-cutters~\cite{rand89}, \cite{bedford91}; the attractor of a contractive IFS amounts then to the repeller of the corresponding cookie-cutter.
We have multivaluedness because~$\pi$ is not $1$--$1$ everywhere; however it is a mild one, as testified by Remark~\ref{ref20}(4) and Theorem~\ref{ref17}.

Let now $\gamma:\bigcupdot_iH_i\to\bigcup_iH_i$ be the map that glues all $H_i$ to the same copy of $\PP^1\Rbb$, and suppose that the following condition holds:
\begin{itemize}
\item[(i)] For every $x\in H_i$ and $x'\in H_j$ such that $\gamma(x)=\gamma(x')$, we have
\[
\gamma\set{B_a\m(x):\text{$B_a\m$ acts on $x$}}=
\gamma\set{B_b\m(x'):\text{$B_b\m$ acts on $x'$}}.
\]
\end{itemize}
This means that the set of $\Fcal$-images of $(i,x)$ equals, up to gluing by $\gamma$, the set of $\Fcal$-images of $(j,x)$. If this happens then $\Fcal$ descends to a Gauss-type multivalued map $F$ on $\bigcup_iH_i$, and we say that the family $\set{I_i}$ constitutes a \newword{realization} of $\Acal$.

All of this can be dualized in the obvious way, and we have a dual multivalued map $\du{\Fcal}$ on $\bigcupdot_iK_i$.
If, at the same time, $\set{I_i}$ is a realization of $\Acal$ and $\set{I_iS}$ a realization of $\du{\Acal}$ (in short, if both $F$ and $\du{F}$ are well defined), then we say that $\set{I_i}$ is a
\newword{geometric} realization of $\Acal$.
\end{definition}

\begin{remark}\label{ref34}
The reader expert in continued fractions will have realized that presenting the pair of dual maps $(\Fcal,\Fcal^\sharp)$ amounts to presenting the natural extension of $\Fcal$. We will not introduce the natural extension, because it is not needed in our proofs, and we already presented much machinery. However, a few words may clarify the situation for the expert; other readers may safely skip the present remark.

The very definition of natural extension refers to the category ---in the sense of category theory--- in which we put ourselves. If we work in the category of measure-preserving systems, then we have much freedom in discarding sets of measure~$0$; this is perfectly fine when seeking, e.g., for absolutely continuous invariant measures.
However, when interested in periodicity issues we cannot discard sets so easily and we have to take into account the topology of the situation. Consider, for example, the tent map on $[0,1]$ and the doubling map on ${[0,1)}\pmod1$. 
With respect to the invariant Lebesgue measure, they
are measure-theoretically the same system, and have the same natural extension. However, in the category of topological dynamical systems they are different, and have completely different natural extensions. In our setting, the dual pair $(\Fcal,\Fcal^\sharp)$ describes the natural extension of $\Fcal$ \emph{without discarding any information about periodic orbits}. The fact ---that may or may not hold--- that the abstract continued fraction $\Acal$ has a geometric realization means that the pair $(\Fcal,\Fcal^\sharp)$ descends to a pair $(F,F^\sharp)$ of maps on subsets of $\PP^1\Rbb$ that nicely code \emph{all} periodic geodesics through a certain section of a hyperbolic surface.
\end{remark}

\begin{remark}\label{ref20}
\begin{enumerate}
\item Every abstract continued fraction has realizations: for example, one can always take pairwise disjoint intervals in the same copy of $\PP^1\Rbb$.
However, unless $n=1$, these realizations are rarely geometric; the key issue here is that we have freedom in choosing the intervals $\set{I_i}$, but once a choice has been made the intervals $\set{I_iS}$ are immutable.
As a matter of fact, we are not claiming that every $\Acal$ has a geometric realization.
\item The notion of a geometric realization is the turning point at which the projective world diverges from the affine one, for the good reason that in the latter no reasonable definition of complementary interval can be given.
\item In all cases treated in this paper a procedure to determine, given $(j,y)$, if $y$ belongs to $H_j$ is readily provided. This yields an easy way of computing~$\Fcal$: on input $(i,x)$ one simply computes the set
\[
\bigl\{(j,B_a\m(x)):a\in E_{ij}\bigr\},
\]
and filters it according to membership in $\bigcupdot_iH_i$. This is not worse than a standard definition by cases.
The version for $\du{\Fcal}$ requires computing
\[
\bigl\{(j,B_a(x)):a\in E_{ji}\bigr\},
\]
and filtering for membership in $\bigcupdot_iK_i$.

If such a procedure is not available (a good case is in~\cite[Example~4.5]{panti22}), then one may still filter by membership in $\bigcupdot_iI_i$ and proceed with the computation. 
Now we have no guarantee that, at any given step, the image set is correct; however, the spurious elements will eventually be filtered out at later steps.
\item If for every $i$ and every pair $a\in E_{ij}$, $b\in E_{ik}$ with $a\not=b$ the intersection $B_a[H_j]\cap B_b[H_k]\subseteq H_i$ is countable, then all $\Fcal$-orbits except at most countably many are ordinary single-valued ones. This is the case for all maps treated in this paper, and it would be interesting to determine sufficient conditions.
\item As usual, we see a sequence $\abf\in E_i^\omega$ as the \newword{symbolic orbit} of the $\Fcal$-orbit of the point $(i,\pi(\abf))$. If $\abf$ is purely periodic, that is of the form $\abf=(a_0\ldots a_{p-1})^\omega$, with $p\ge1$ minimal with this property, then we say that $a_0\ldots a_{p-1}$ is the \newword{symbolic period} of the orbit.
\item The issue of defining Gauss-type maps at discontinuity points is recurrent in the literature, even in the single-interval setting, and various ways of dealing with it have been proposed, such as devising a selection procedure, or leaving maps undefined at discontinuity points. This is usually harmless in a measure-theoretic context, but not so in ours; see, e.g., Example~\ref{ref32}.
Thus, in Definition~\ref{ref16} we 
accept multivalued maps, but require that the multivaluedness be intrinsic to the combinatorial structure, not the result of a weird realization.
\item Changing $I_i$ to $I_i'$ amounts to transforming $H_i$ and $K_i$ to $I_i' I_i\m[H_i]$ and $I_i' I_i\m[K_i]$, respectively, the other components of the attractors being unaffected.
\item We recap the direction of arrows:
\begin{itemize}
\item The branches $B_{i\sigma j}\m$ of $F$, the maps $D_{i\sigma j}$ of the dual IFS, and the $F$-symbolic sequences follow the direction $i\to j$ of edges in $\Gcal_\Acal$.
\item The branches $D_{i\sigma j}\m$ of $\du{F}$, the maps $B_{i\sigma j}$ of the base IFS, and the $\du{F}$-symbolic sequences go against the direction of edges in $\Gcal_\Acal$.
\end{itemize}
\end{enumerate}
\end{remark}

The following combinatorial fact is our main technical tool.

\begin{theorem}
Let\label{ref17} $\Acal$ be an abstract continued fraction, let $\set{I_i}$ be unimodular intervals, and $(\Gcal_\Acal,\set{B_a})$ the graph-directed IFS thus determined.
\begin{enumerate}
\item There exists a number $M\ge1$, depending on $\Acal$ only, such that for every $x_0\in H_{i_0}$ the cardinality of the fiber $\pi\m(x_0)\subseteq E_{i_0}^\omega$ is bounded by $M$.
\item If the fiber $\pi\m(x_0)$ contains a purely periodic sequence, then it has cardinality $1$.
\end{enumerate}
\end{theorem}

Theorem~\ref{ref17} guarantees that the multivaluedness of our maps $\Fcal,\Fcal^\sharp$ is not only algorithmically treatable (see Remark~\ref{ref20}(3)), but definitely limited: no point has more than $M$ forward orbits. Moreover, purely periodic points will never give any trouble: if one of these orbits happens to be purely periodic, then there are no other orbits at all.
We postpone the proof of Theorem~\ref{ref17} after providing two overdue examples.

\section{Two examples}\label{ref13}

Our first example on the Farey map is an easy warmup for much more demanding second one.

The Farey map is the realization of the abstract continued fraction $\Acal=\set{0l0,0nf0}$ obtained by taking $I_0=[0,1]$, which identifies with the matrix $L$ according to our conventions. Since only one node is involved, this is automatically a geometric realization. Note indeed that in the one-node case all realizations of the same $\Acal$ are clearly conjugate, whereas in the multiple-node case different realizations may not be conjugate; see Remark~\ref{ref36} on our second example. Setting for short $a=0l0$ and $b=0nf0$, we have $B_a=LLL\m=L$ and $B_b=LNFL\m=\bbmatrix{1}{-1}{-1}{}$. Since
\[
B_a[I_0]\cup B_b[I_0]=[0,1/2]\cup[1/2,1]=I_0,
\]
the fixed point of the Hutchinson operator is $H_0=I_0$, and $F$ is the well known Farey map, induced by $B_a\m$ on $[0,1/2]$ and by $B_b\m$ on $[1/2,1]$.

Now, according to the definition before Remark~\ref{ref30}, we have $D_a=(I_0S)\Proj(l^\sharp)(I_0S)\m=LSN(LS)\m=L\m$ and
$D_b=(I_0S)\Proj\bigl((nf)^\sharp\bigr)(I_0S)\m=LSNF(LS)\m=
\bbmatrix{1}{-1}{-1}{}$.
The Hutchinson operator for $(\Gcal_{\du{\Acal}},\set{D_a,D_b})$ immediately shrinks the complementary interval $[1,0]$ to $[\infty,0]$, and then fixes the latter; therefore the dual attractor is $K_0=[\infty,0]$, and Theorem~\ref{ref3} yields Corollary~\ref{ref4}(1). Moreover, $F^\sharp$ is induced by $D_a\m$ on $[-1,0]$ and by $D_b\m$ on $[\infty,-1]$.

As it is well known, Schweiger's jump transformation
(see~\cite[Chapter~18]{schweiger95}; we will treat the construction in more detail in~\S\ref{ref26}) accelerates the Farey map to the Gauss one.
Indeed, let $J=I_0\setminus[0,1/2]$ be the complement of the set of points on which $F$ is defined by the parabolic branch $B_a\m$.
Setting $e(x)=\min\set{t\ge0:F^t(x)\in J}$ then $F_{\jump}$, defined by $F_{\jump}(x)=F^{e(x)+1}(x)$ is the Gauss map. 
In Theorem~\ref{ref31} we will show that, in cases such as the present one, 
$F_{\jump}$ is dual to the acceleration $F^\sharp_R$ of $F^\sharp$ on the image $R=D_b[K_0]$ of $K_0$ under the non-parabolic letters. In this case $D_b[K_0]=[\infty,-1]$, and it is readily seen that $F^\sharp_R$ is conjugate to $F_{\jump}$ via the involution $x\mapsto x\m$; we thus recover Galois's Theorem~\ref{ref5}.

Our second example is much more involved, but has the advantage of presenting all the difficulties we have to face when multiple intervals are involved, in particular the possibility of attractors $H_i$ strictly contained in $I_i$, of large overlappings of the various $I_iS$, of overlappings at infinity, and of non-geometric realizations. Its combinatorial structure will reappear in Example~\ref{ref39}.

Let $n=2$ and $\Acal=\set{0n0, 1l1, 0nl0, 0nnl1, 1nf0, 1lnf1}$,
whose elements are denoted $o,p,q,r,s,t$, in this order; the letters $o$ and $p$ are parabolic.
The graph $\Gcal_\Acal$ is displayed in Figure~\ref{fig1}.
\begin{figure}[ht]
\begin{tikzpicture}[scale=1.6]
\node (0) at (0,0)  []  {$0$};
\node (1)  at (2,0) []  {$1$};
\path
(0) edge[Nedge,out=110,in=160,loop] node[left] {$q=nl$}
(0) edge[Nedge,out=-160,in=-110,looseness=8,loop] node[left] {$o=n$}
(0) edge[Nedge,out=-30,in=-150] node[below,inner sep=1pt] {$r=nnl$} (1)
(1) edge[Nedge,out=-70,in=-20,loop] node[right] {$p=l$}
(1) edge[Nedge,out=20,in=70,looseness=8,loop] node[right] {$t=lnf$}
(1) edge[Nedge,out=150,in=30] node[above,inner sep=1pt] {$s=nf$} (0);
\end{tikzpicture}
\caption{The graph $\Gcal_\Acal$; $s=nf$ is shorthand for $s=1nf0$, and similarly for the other edges}
\label{fig1}
\end{figure}
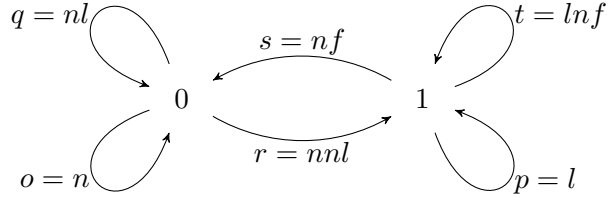

We have $G_\Acal=8\m\ppmatrix{6}{1}{4}{6}$, which has eigenvalue $1$ with right eigenvector $(1\;2)\transpose$, the other eigenvalue being $1/2$. We fix the intervals $I_0=[-1,0]$ and $I_1=[0,1]$, corresponding to $\bbmatrix{}{-1}{1}{1}$ and $L$, respectively.
Our formalism reduces the computation of the matrices $B_{i\sigma j}$ to converting lowercase to uppercase and replacing nodes with interval matrices. For example, $B_{1nf0}$ equals $I_1NFI_0\m=\bbmatrix{}{1}{1}{2}$.
It is instructive to present the conjugated affine IFS $(\Gcal_\Acal,\set{\Cbf_a})$ on the same intervals; letting $\Ibf_i$ be the only affine order-preserving map that sends $[0,1]$ to $I_i$, we have 
$\Cbf_{1nf0}=\Ibf_1\Nbf\Fbf\Ibf_0\m(x)=-x/2+1/2$. As $I_0$ and $I_1$ overlap only at $0$, we can easily draw the two g-d IFS, see Figure~\ref{fig2}.
\begin{figure}[ht]
\includegraphics[width=5cm]{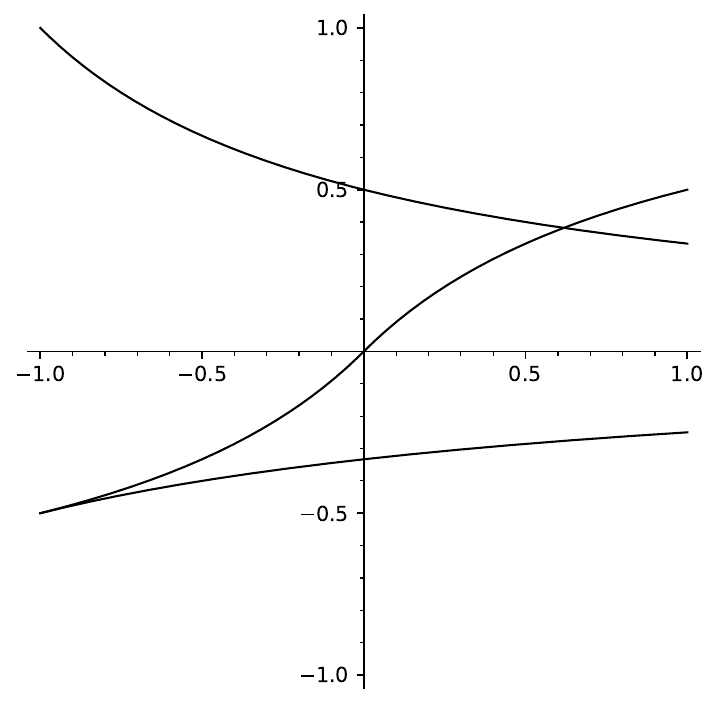}
\hspace{0.6cm}
\includegraphics[width=5cm]{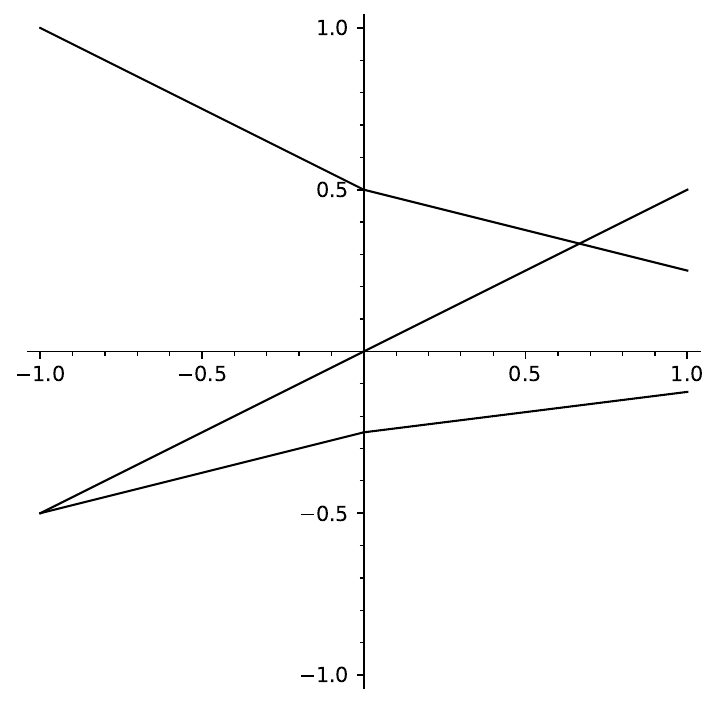}
\caption{The projective and the affine realizations of $\Acal$}
\label{fig2}
\end{figure}

The six maps in $\set{B_a:a\in\Acal}$ match in pairs, so the end result appears to be an ordinary IFS. However, this is somehow misleading: $B_o$ and $B_p$ match at $(0,0)$ as functions, but clearly they are different parabolic matrices.
On the other hand, $B_s$ and $B_t$ match at $(0,1/2)$ and are the same matrix $\bbmatrix{}{1}{1}{2}$; analogously for $B_q$ and~$B_r$.

The attractor of $(\Gcal_\Acal,\set{B_a})$ is easily guessed and verified to be $[\tau-2,0]\cupdot[0,\tau-1]$, with $\tau$ the golden ratio, while the attractor of $(\Gcal_\Acal,\set{\Cbf_a})$ is $[-1/3,0]\cupdot[0,2/3]$. Note that, indeed, the Minkowski homeomorphism sends $\tau-1$ to $2/3$, while its conjugate by 
translation-by-$1$ sends $\tau-2$ to $-1/3$.
Note also that the ratio of the Lebesgue measures of $[-1/3,0]$ and $[0,2/3]$ is $1/2$, which fits with the $1$-eigenvector of $G_\Acal$ being $(1\;2)\transpose$.
Finally, note that taking $U_0=(\tau-2,0)$ and $U_1=(0,\tau-1)$, we see directly that $(\Gcal_\Acal,\set{B_a})$ satisfies the Open Set Condition so that, by Lemma~\ref{ref14}, $\Acal$ and $\du{\Acal}$ are indeed codes and abstract continued fractions.

Since $H_0=[\tau-2,0]$ and $H_1=[0,\tau-1]$ intersect only at $0$, which is fixed by both the matrices $B_o\m$ and $B_p\m$ acting on it, this is a realization of $\Acal$.
The graph of the Gauss-type map $F$ on $H_0\cup H_1$ is thus obtained by reflecting Figure~\ref{fig2} left along the diagonal, and restricting it to $[\tau-2,\tau-1]^2$. 
The resulting map, shown in Figure~\ref{fig3} left, may look unfamiliar, but becomes familiar once we apply the jump transformation to it. Indeed, let $J=[(\tau-3)/5,-\tau+2]^c$ be the complement of the set of points on which $F$ is defined by the two parabolic branches.
Then $F_{\jump}$, defined as in the example of the Farey map and shown in Figure~\ref{fig3} right, is nothing else than the $\alpha$-continued fraction of~\cite{nakada-natsui08}, \cite{kraaikamp-schmidt-steiner12}, \cite{carminati-tiozzo13}, for $\alpha=\tau-1$.
\begin{figure}[ht]
\includegraphics[width=5cm]{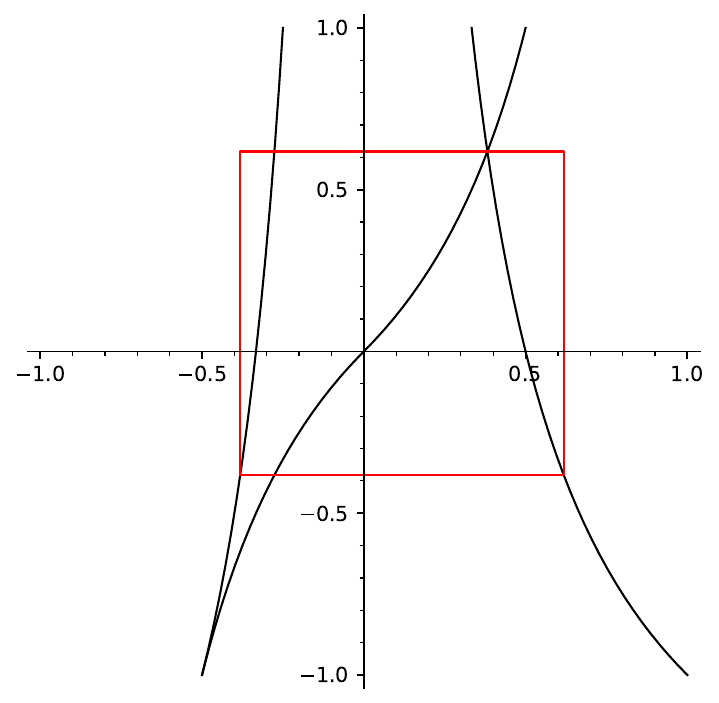}
\hspace{0.6cm}
\includegraphics[width=5cm]{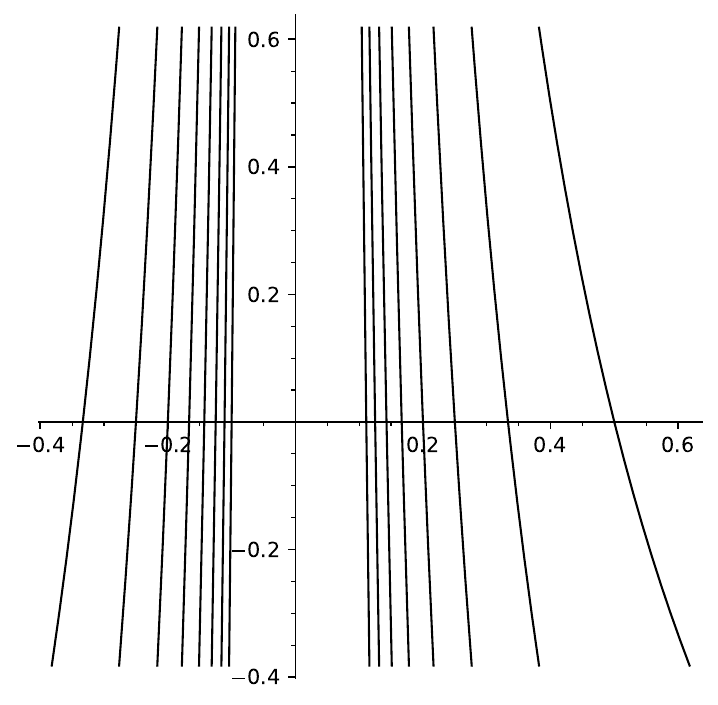}
\caption{Left: the graph of $F$. The branches, listed from left to right in the order they touch the $x$-axis, are $B_q\m,B_r\m,B_o\m,B_p\m,B_t\m,B_s\m$. Right: the graph of $F_{\jump}$}
\label{fig3}
\end{figure}

We now discuss $\du{\Acal}$, which is more delicate. The intervals $I_0S=[0,-1]$ and $I_1S=[1,0]$ intersect in $[1,-1]\cup\set{0}$; thus, working directly with them is inconvenient.
It is much simpler to use the approach of Remark~\ref{ref21} to compute the attractor of the g-d IFS determined by $\du{\Acal}$ on two disjoint copies of $[0,1]$, and then apply Remark~\ref{ref20}(7). After some practice the process is automatic; we present the end result in the following lemma, in a form that is adaptable to all cases concerning us.

\begin{lemma}
Define\label{ref35}
\begin{align*}
R_0& =[\infty,-2],& R_1& =[\infty,-2],\\
K_0& =\bigcup_{t\ge0}D_o^tR_0\cup\set{0},&
K_1& =\bigcup_{t\ge0}D_p^tR_1\cup\set{0}.
\end{align*}
Then we have
\begin{align}
R_0&=D_qK_0\cup D_sK_1,& R_1&=D_rK_0\cup D_tK_1,\label{eq9}\\
K_0&=D_oK_0\cup D_qK_0\cup D_sK_1,&
K_1&=D_rK_0\cup D_tK_1\cup D_pK_1.\label{eq10}
\end{align}
By the last two equalities, the pair $(K_0,K_1)$ is a fixed point for the Hutchinson operator, and thus is the attractor of the g-d IFS $(\Gcal_{\du{\Acal}},\set{D_a})$.
\end{lemma}
\begin{proof}
The matrices $B_r\m$ and $B_q\m$ are equal and send $0$ to $\infty$ in an orientation-preserving way; analogously $B_s\m$ and $B_t\m$ are equal and send $0$ to $\infty$ in an orientation-reversing way. Therefore 
the matrix identities
\begin{equation}\label{eq11}
N=B_r\m B_oB_r=B_q\m B_oB_q=B_s\m B_pB_s=B_t\m B_pB_t
\end{equation}
hold (of course, they can also be checked explicitly). Remembering that $D_a=B_a\m$ for every $a\in\Acal$, we obtain
\begin{align*}
N\m D_r&=D_rD_o,&  N\m D_q&=D_qD_o,\\
N\m D_s&=D_sD_p,& N\m D_t&=D_tD_p.
\end{align*}
We now observe that $D_qR_0\cup D_sR_1=[-3,-5/2]\cup[-5/2,-2]=[-3,-2]$ and compute
\begin{equation*}
\begin{split}
D_qK_0\cup D_sK_1 &= D_q\Bigl[\bigcup_{t\ge0}D_o^tR_0
\cup\set{0}\Bigl]
\cup D_s\Bigl[\bigcup_{t\ge0}D_p^tR_1
\cup\set{0}\Bigl]\\
&= \bigcup_{t\ge0}N^{-t}D_qR_0
\cup\bigcup_{t\ge0}N^{-t}D_sR_1\cup\set{\infty}\\
&= \bigcup_{t\ge0}N^{-t}[-3,-2]\cup\set{\infty}\\
&= R_0.
\end{split}
\end{equation*}
An analogous manipulation shows the second identity in~\eqref{eq9}, and the identities in~\eqref{eq10} follow easily.
\end{proof}

We visualize the situation by reparametrizing the projective line as the border of the Poincar\'e disk with parameter the argument, ranging in $[-\pi/2,3\pi/2]$. Thus, in Figure~\ref{fig4} left the horizontal and vertical segments in the lower left corner range from $-\pi/2$ to 
the argument of the image of $-2$ under stereographic projection through~$i$, and represent $[0,-2]$, while those in the upper right corner range from $\pi/2$ to $3\pi/2$, 
and represent $[\infty,0]$. The inner rectangles show part of $\bigl(K_0\cup K_1\bigr)^2$; there are countably many more rectangles, accumulating to the borders.
\begin{figure}[ht]
\includegraphics[width=5.5cm]{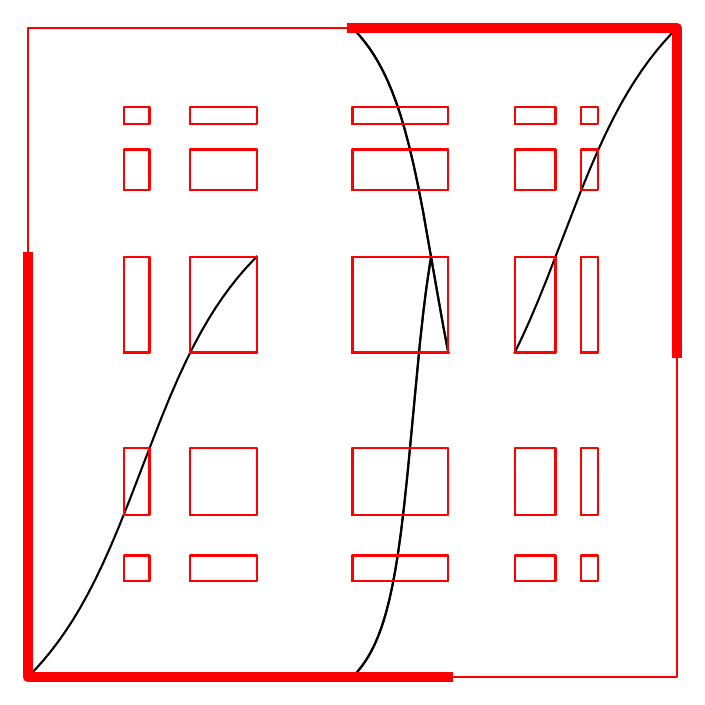}
\hspace{0.6cm}
\includegraphics[width=5.5cm]{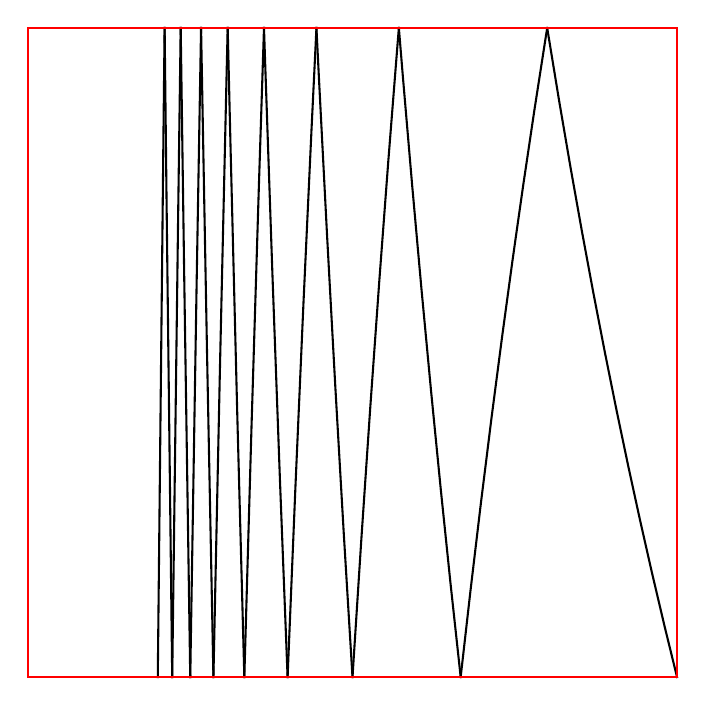}
\caption{Left: the graph of $\du{F}$. The increasing branches are $D_o\m$, $D_q\m=D_r\m$, and $D_p\m$, listed from left to right. The decreasing one is $D_s\m=D_t\m$. Right: the first-return map $\du{F}_{[\infty,-2]}$}
\label{fig4}
\end{figure}

For $i\sigma j$ ranging in $\Acal$, the branches of $\du{F}$
\begin{multline*}
\bigl\{\bigl(x,D_{0\sigma j}\m(x)\bigr):0\sigma j\in\Acal\text{ and }x\in D_{0\sigma j}[0,-2]\bigr\}\\
\cup
\bigl\{\bigl(x,D_{1\sigma j}\m(x)\bigr):1\sigma j\in\Acal\text{ and }x\in D_{1\sigma j}[\infty,0]\bigr\}
\end{multline*}
are also shown (we are talking as if we already knew that we are dealing with a geometric realization; this simplifies things and will be verified in the next paragraph).
Only four branches appear, rather than six, because, as noted in discussing Figure~\ref{fig2}, the four nonparabolic matrices collapse in pairs.
The map $\du{F}$ is piecewise-defined by these four branches.
Only $D_o\m=L\m$ acts on $K_0\setminus[\infty,2]$, fixing $0$ and pushing each $D_o^t[\infty,-2]$, for $t\ge1$, to $D_o^{t-1}[\infty,-2]$; analogously only $D_p\m=L$ acts on $K_1\setminus[\infty,2]$.
On the other hand, the two nonparabolic branches act on $[\infty,-2]$, clearly in an alternating way.

In order to show that we are dealing with a geometric realization, so that $\du{F}$ is indeed well defined,
we have to check the points in $K_0\cap K_1=[\infty,-2]\cup\set{0}$; as for $F$, the point $0$ does not give trouble. On the other hand,
as shown in the proof of Lemma~\ref{ref35}, $[\infty,-2]$ is the union of the two sets
\begin{equation*}\label{eq1}
\begin{split}
D_q[K_0]&=D_r[K_0]
=\bigcup_{t\ge0}N^{-t}[-3,-5/2]\cup\set{\infty}=A,\\
D_s[K_1]&=D_t[K_1]
=\bigcup_{t\ge0}N^{-t}[-5/2,-2]\cup\set{\infty}=B.
\end{split}
\end{equation*}
If $x$ belongs to the interior of $A$, then its $\du{F}$-value equals $D_q\m(x)$ if we consider $x$ to sit in $K_0$, and equals $D_r\m(x)$ ---which is the same number--- if we consider $x$ to sit in $K_1$; thus $\du{F}$ is single-valued at $x$, and similarly if $x$ belongs to the interior of $B$, or equals one of the two points $\infty,-5/2$ at which the two branches meet. On the other hand, if $x$ belongs to $A\cap B\setminus\set{\infty,-5/2}$
then $\du{F}$ is double-valued at $x$, but this ambiguity does not depend on our considering $x$ in $K_0$ or in $K_1$. Indeed, if $x$ is thought in $K_0$, then its two distinct images are $D_q\m(x)$ and $D_s\m(x)$, while if it is thought in $K_1$ it has images $D_t\m(x)$ and $D_r\m(x)$, the same pair of numbers.

Again, the map $\du{F}$ is not as weird as it appears. Let $\du{F}_{[\infty,-2]}$ be the induced first-return map on $[\infty,-2]$, namely $\du{F}_{[\infty,-2]}(x)=(\du{F})^{q(x)}(x)$, with $q(x)=\min\set{t\ge1:(\du{F})^t(x)\in[\infty,-2]}$. 
This is an ordinary single-valued map, because the possibly distinct $\du{F}$-images of the same point glue together when coming back to $[\infty,-2]$; see Figure~\ref{fig4} right.
Theorem~\ref{ref22} will show that  
$\du{F}_{[\infty,-2]}$ is dual to $F_{\jump}$.
On the other hand, by
direct inspection one sees that $\du{F}_{[\infty,-2]}$ is conjugate via $x\mapsto -x\m$ to a familiar map, namely 
the folded version on $[0,1/2]$ of the Nearest Integer continued fraction
\[
x\mapsto\bigl\vert x\m-\text{(the integer nearest to $x\m$)}\bigr\vert.
\]

\begin{remark}
Suppose\label{ref36} we change $I_0=[-1,0]$ to $I_0'=[0,1]=NI_0=I_1$; as noted in Remark~\ref{ref20}(7), the new attractors are $(H_0',H_1)$ and $(K_0',K_1)$, with 
$H_0'=NH_0=[\tau-1,1]$ and $K_0'=NK_0$.
The intersection of $H_0'$ and $H_1$ is now $\set{\tau-1}$, but no issue arises, $F$ remains well defined (see Figure~\ref{fig8}) and we still have a realization.
\begin{figure}[ht]
\includegraphics[width=5.5cm]{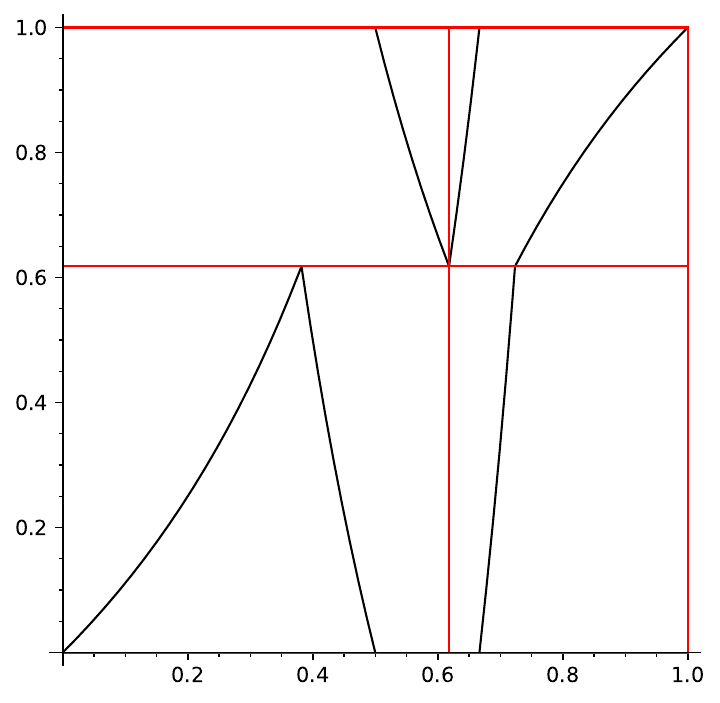}
\caption{A non-geometric realization of $\Acal$}
\label{fig8}
\end{figure}
However, let $\alpha$ be the real cube root of $13$ and consider the point $\alpha-49/9=-3.0931\ldots\in K_0'\cap K_1$. Applying the algorithm in Remark~\ref{ref20}(3), the 
$\du{\Fcal}$-orbit of $(0,\alpha-49/9)$ begins with
\begin{gather*}
(0, \alpha - 49/9),\\
(1, -729/54523\,\alpha^2 - 3240/54523\,\alpha - 14400/54523),\\
(1, -729/20314\,\alpha^2 - 2511/20314\,\alpha - 8649/20314),
\end{gather*}
while that of $(1,\alpha-49/9)$ with
\begin{gather*}
(1, \alpha - 49/9),\\
(1, -729/20314\,\alpha^2 - 2511/20314\,\alpha - 8649/20314),\\
(1, -729/1171\,\alpha^2 - 1782/1171\,\alpha - 4356/1171),
\end{gather*}
so there is a lag of one time step between the two orbits,
$\du{\Fcal}$ does not project to a Gauss-type map, and $I_0'=I_1=[0,1]$ is not a geometric realization of $\Acal$.
\end{remark}

One of the referees of this paper suggested that, at least for $\alpha$-continued fractions, the fact that a realization is or is not geometric should be related to matching~\cite{carminati_et_al10}, \cite{kalle_et_al20}. We agree, and consider the issue an interesting topic for further research.

\section{Proof of Theorem~\ref{ref17}}\label{ref24}

We start by proving (1).
The product operation in $\widetilde\Sigma$ can be extended to cover infinite products. Namely, given an infinite sequence 
$\abf=(i_0\sigma_0 i_1)(i_1\sigma_1 i_2)\ldots$ of elements of $\widetilde\Sigma$, in which every pair of consecutive terms can be multiplied, we consider, for every $t\ge1$, the product $i_0w_ti_t$ of the first~$t$ terms, in which the word $w_t$ over $\set{l,n,f}$ is simplified according to the rules (so that it contains at most one occurrence of $f$, precisely at the word end). It is then clear that, for $t$ going to infinity, the sequence $i_0w_ti_t$ converges positionwise to a uniquely determined sequence $i_0\zbf\in\set{0,\ldots,n-1}\times\set{l,n}^\omega$, which we regard as the infinite product of $\abf$ and denote by $\Phi(\abf)$.
The sequence $\zbf$ thus obtained is an infinite path in the graph with one node and edges $l,n$. Letting that node correspond to the unimodular interval $[0,\infty]$, the construction in Lemma~\ref{ref27} yields a point $\pi^*(\zbf)\in[0,\infty]$, and one easily proves that $\pi(\abf)=I_{i_0}\bigl(\pi^*(\zbf)\bigr)\in I_{i_0}$.

It is well known that for every $x\in[0,\infty]$ there exist precisely one (if $x$ is irrational, $0$ or $\infty$), or precisely two (if $x$ is rational different from $0$ and $\infty$), sequences $\zbf\in\set{l,n}^\omega$ such that $\pi^*(\zbf)=x$. Statement (1) of Theorem~\ref{ref17} will then follow once we prove that, given $i_0\zbf$, the number of paths $\abf\in E_{i_0}^\omega$ such that $\Phi(\abf)=i_0\zbf$ is uniformly bounded. We show this fact by constructing a nondeterministic transducer $\Tcal_\Acal$ that, on input $\zbf$ from node $i_0$, outputs in parallel the set $\set{\abf\in E_{i_0}^\omega:\Phi(\abf)=i_0\zbf}$. 
The construction of $\Tcal_\Acal$ from $\Acal$ is the $n$-nodes generalization of the $1$-node version introduced in~\cite[\S5]{panti18}, and works as follows. We let $z$ vary in $\set{l,n}$ and let ${}'$ be the bijection on $\set{l,n}^{<\omega}$ that exchanges $l$ with $n$ componentwise:
we also introduce $n$ new nodes $0',\ldots,(n-1)'$.
\begin{enumerate}
\item Given $i\sigma j\in\Acal$, we consider all \newword{splittings} $i\sigma j=(iu)(vj)$ such that $u$ is a possibly empty word in $\set{l,n}$ and $v$ is a nonempty word in $\set{l,n,f}$, containing at most one occurrence of $f$, necessarily at the end but not at the beginning. For example, $1nlf1$ has two splittings, $(1)(nlf1)$ and $(1n)(lf1)$.
We call $iu$ a \newword{prefix} and $i'u'$ its \newword{dual prefix}; the set of nodes of $\Tcal_\Acal$ is given by all prefixes and dual prefixes.
\item We add edges to $\Tcal_\Acal$ as follows.
\begin{itemize}
\item[(2$'$)] For each pair of prefixes of the form $iu$, $iuz$, we add a directed edge labeled $z$ from $iu$ to $iuz$, as well as one labeled $z'$ from $i'u'$ to $i'u'z'$.
\item[(2$''$)] For each splitting of the form $(iu)(zj)=i\sigma j\in\Acal$, we add an edge labeled $z\vert i\sigma j$ from $iu$ to $j$, as well as an edge labeled $z'\vert i\sigma j$ from $i'u'$ to $j'$.
\item[(2$'''$)] For each splitting of the form $(iu)(zfj)=i\sigma j\in\Acal$, we add an edge labeled $z\vert i\sigma j$ from $iu$ to $j'$, as well as an edge labeled $z'\vert i\sigma j$ from $i'u'$ to $j$.
\end{itemize}
\end{enumerate}

A little thinking shows ---and the formal proof in~\cite[Lemma 5.1]{panti18} for the $1$-node case extends without difficulties--- that feeding $\Tcal_\Acal$ with input $\zbf\in\set{l,n}^\omega$ from node $i_0$ produces all paths $\abf$ in $\Gcal_\Acal$ such that $\Phi(\abf)=i_0\zbf$. The action of $\Tcal_\Acal$ is straightforward: we put a token at node $i_0$ and start reading $\zbf=z_0z_1z_2\ldots$. If an edge labeled $z_0$ or $z_0\vert a$ leaves $i_0$ we move the token to the target node, outputting $a\in\Acal$ in case the edge is labeled $z_0\vert a$. 
We then check if an edge labeled $z_1$ leaves the new node, and repeat.
The transducer is nondeterministic, meaning that more than one edge with the same label may leave a node; in this case, we split the token as necessary and continue the computation in parallel.
If at some time $t$ a token sits in a state from which no edge labeled $z_t$ leaves, it disappears. The computation fails if all tokens eventually disappear; this means that $E^\omega\cap\Phi\m(i_0\zbf)=\emptyset$. Otherwise it runs forever, outputting $E^\omega\cap\Phi\m(i_0\zbf)$ in parallel.

The key issue now is that no node will ever host two tokens. Indeed, if so, then at some future time $t$ the two tokens will end up in a node labeled either $j$ or $j'$, and the words $w,w'\in\Acal^{<\omega}$ output by the tokens up to time $t$ will be different as words, but both equal either to $i_0z_0\ldots z_t j$ (if they end up at $j$) or to $i_0z_0\ldots z_tf j$ (if they end up at $j'$) as elements of $\widetilde\Sigma$, contradicting the fact that $\Acal$ is a code. Thus, for every $i_0\zbf$, the cardinality of $E^\omega\cap\Phi\m(i_0\zbf)$ is bounded by the number $M$ of nodes of $\Tcal_\Acal$, as requested by Theorem~\ref{ref17}(1).

\begin{example}
Let\label{ref39} $\Acal$, $I_0$, $I_1$, $o,p,q,r,s,t$ be as in the second example of~\S\ref{ref13}. Figure~\ref{fig6} shows the transducer $\Tcal_\Acal$, omitting the labeling of the intermediate nodes.
\begin{figure}[ht]
\begin{tikzpicture}[scale=1.4]
\node (0) at (0,3)  []  {$0$};
\node (1)  at (4,3) []  {$1$};
\node (0p) at (0,1)  []  {$0'$};
\node (1p)  at (4,1) []  {$1'$};
\node (a)  at (1,4) []  {$\circ$};
\node (b)  at (3,4) []  {$\circ$};
\node (c)  at (4,2) []  {$\circ$};
\node (d)  at (5,2) []  {$\circ$};
\node (e)  at (1,0) []  {$\circ$};
\node (f)  at (3,0) []  {$\circ$};

\path
(0) edge[Nedge,out=150,in=210,loop] node[left] {$n\vert o$}
(0) edge[Nedge,out=80,in=190] node[left] {$n$} (a)
(a) edge[Nedge,out=260,in=10] node[right] {$l\vert q$} (0)
(a) edge[Nedge] node[above] {$n$} (b)
(b) edge[Nedge] node[right] {$l\vert r$} (1)
(1) edge[Nedge,out=-30,in=30,loop] node[right] {$l\vert p$}
(1) edge[Nedge] node[below,pos=.65] {$n\vert s$} (0p)
(1p) edge[Nedge] node[above,pos=.65] {$l\vert s$} (0)
(1) edge[Nedge] node[right] {$l$} (d)
(d) edge[Nedge] node[right] {$n\vert t$} (1p)
(1p) edge[Nedge,out=30,in=-30,loop] node[right] {$n\vert p$}
(1p) edge[Nedge] node[left] {$n$} (c)
(c) edge[Nedge] node[left] {$l\vert t$} (1)
(0p) edge[Nedge,out=210,in=150,loop] node[left] {$l\vert o$}
(0p) edge[Nedge,out=280,in=170] node[left] {$l$} (e)
(e) edge[Nedge,out=110,in=-10] node[right] {$n\vert q$} (0p)
(e) edge[Nedge] node[below] {$l$} (f)
(f) edge[Nedge] node[right] {$n\vert r$} (1p)
;
\end{tikzpicture}
\caption{The transducer $\Tcal_\Acal$, for the abstract cf $\Acal$  in \S\ref{ref13}}
\label{fig6}
\end{figure}
The point $(-\tau+8)/11\sim0.58017\ldots$ is in $I_1$ and equals $I_1(-\tau+3)$. Let $\zbf=nl(ln)^\omega$; then $\pi^*(\zbf)=-\tau+3$, since $(NL)\m(-\tau+3)=\tau-1$, which is fixed by $LN=\bbmatrix{1}{1}{1}{2}$. Feeding $\Tcal_\Acal$ with~$\zbf$ from node $1$, it outputs the two infinite sequences $soq^\omega$ and $srsq^\omega$, which are indeed the two symbolic orbits of $(-\tau+8)/11$ under the map of Figure~\ref{fig3} left.
\end{example}

We now prove Theorem~\ref{ref17}(2). 
Let $\abf,\bbf\in E_{i_0}^\omega$, with
$\abf=(a_0\ldots a_{p-1})^\omega$ purely periodic, be such that $\Phi(\abf)=\Phi(\bbf)$.
We will prove $\abf=\bbf$ by adapting the argument given in~\cite[Lemma~1.4]{devolder_et_al94} for the case of ordinary binary codes.

It is enough to show that $a_0=b_0$; indeed, if so, then $\Phi\bigl((a_1\ldots a_{p-1}a_0)^\omega\bigr)=\Phi(b_1b_2b_3\ldots)$ 
(because both correspond to $B_{a_0}\m(x_0)\in I_{i_1}$),
and we are through by induction. By possibly replacing $a_0\ldots a_{p-1}$ with its square, we may assume that $a_0\ldots a_{p-1}=i_0w i_0$ in $\widetilde\Sigma$, with $w\in\set{l,n}^{<\omega}$; thus $\Phi(\abf)=\Phi(\bbf)=
i_0w^\omega$. Since the number of nodes and that of prefixes of $w$ are both finite, there must exist a proper prefix $u$ of $w$, possibly empty, a node $j$, and indices $k>h\ge0$, $r\ge0$, $s\ge1$, such that
\begin{align*}
b_0\ldots b_{h-1}&=i_0 w^r uj,\\
b_0\ldots b_{k-1}&=i_0 w^{r+s} uj.
\end{align*}
But then
\[
b_0\ldots b_{k-1}=i_0w^si_0i_0w^ruj=(a_0\ldots a_{p-1})^sb_0\ldots b_{h-1}
\]
in $\widetilde\Sigma$; since $\Acal$ is a code, $b_0=a_0$ as desired.

\section{The Galois theorem}\label{ref25}

We now state our full version of Theorem~\ref{ref3}; the proof is crucially based on Theorem~\ref{ref17}.

\begin{theorem}
Let\label{ref19} $\set{I_i}$ be a realization of the abstract continued fraction $\Acal$, and let $\set{H_i}$, $\set{K_i}$, $\Fcal$, $F$ be as in Definition~\ref{ref16}.
Given $\omega_0\in(\bigcup_iH_i)\setminus\Qbb$, the following statements are equivalent.
\begin{itemize}
\item[(i)] At least one of the $F$-orbits of $\omega_0$ is purely periodic.
\item[(ii)] $\omega_0$ has precisely one $F$-orbit, and that orbit is purely periodic.
\item[(iii)] $\omega_0$ is a quadratic irrational and for some $i_0$ we have $\omega_0\in H_{i_0}$ and $\omega_0'\in K_{i_0}$.
\end{itemize}
If this happens, then the symbolic periods of $(i_0,\omega_0)$ under $\Fcal$ and of $(i_0,\omega_0')$ under $\du{\Fcal}$ are unique and the reverse of each other. If the realization is geometric then $\omega_0'$ has a unique purely periodic $\du{F}$-orbit, of the same period $p$ as the $F$-orbit of $\omega$, and the identity $(\du{F})^t(\omega'_0)=\bigl(F^{-t}(\omega_0)\bigr)'$ holds for every $t\pmod{p}$.
\end{theorem}
\begin{proof}
Assume (i) and let $\omega_0,\omega_1,\ldots$ be the given purely periodic $F$-orbit of $\omega_0$. Then there exists $i_0$ and a purely periodic $\abf\in E_{i_0}^\omega$ such that $\omega_t=\pi(S^t\abf)$ for every~$t$. By definition, any other $F$-orbit $x_0,x_1,\ldots$ 
of $\omega_0=x_0$ is of the form $x_t=\pi(S^t\bbf)$ for some $j_0$ and $\bbf\in E_{j_0}^\omega$.

\paragraph{\emph{Claim}} Given such an $F$-orbit, there exists $\cbf\in E_{i_0}^\omega$ such that $\pi(S^t\cbf)=\pi(S^t\bbf)$ for every $t$.
\paragraph{\emph{Proof of Claim}} 
Since $\set{I_i}$ is a realization of $\Acal$, by definition there exists $\cbf^0\in E_{i_0}^\omega$ with $\pi(\cbf^0)=x_0$ and $\pi(S\cbf^0)=x_1$. Analogously, letting $i_1$ be the starting node of $S\cbf^0$, there exists $\cbf^1\in E_{i_1}^\omega$ with $\pi(\cbf_1)=x_1$ and $\pi(S\cbf^1)=x_2$, and so on by induction. By construction, the sequence $\cbf=c^0_0c^1_0c^2_0\ldots$ of first elements of the successive $\cbf^t$ belongs to $E_{i_0}^\omega$. Moreover, for each $t\ge0$, we have
\[
x_t=\pi(\cbf^t)=B_{c^t_0}\bigl(\pi(S\cbf^t)\bigr)=
B_{c^t_0}(x_{t+1})=B_{c^t_0}B_{c^{t+1}_0}(x_{t+2})=\cdots.
\]
This shows that $x_t=\pi(S^t\cbf)$, as requested.

Having proved our claim, we conclude that 
$\pi(\abf)=\pi(\cbf)$, with $\abf$ purely periodic.
By Theorem~\ref{ref17}(2) we have $\abf=\cbf$, and thus $\omega_t=x_t$ for every~$t$; this establishes~(ii).

Assume (ii), and let $\abf=(a_0\ldots a_{p-1})^\omega \in E_{i_0}^\omega$ be such that the given orbit is
the $\pi$-image of the $S$-orbit of $\abf$.
Writing $a_t=i_t\sigma_t i_{t+1}$, we must then have $i_p=i_0$. Since $\omega_0$ is fixed by $B_{a_0\ldots a_{p-1}}$, 
which is not the identity matrix by the observation after Lemma~\ref{ref14}, $\omega_0$
is a quadratic irrational. Applying the Galois conjugation and Remark~\ref{ref30} we obtain $D_{a_{p-1}\ldots a_0}(\omega_0')=\omega_0'$, which is in $K_{i_p}=K_{i_0}$; thus (iii) holds.

Let a discriminant $D\in\Zbb\pp$ be given; then the set of all pairs $(\omega,\omega')$ of Galois conjugates of quadratic irrationals that have discriminant $D$ and are such that $\omega'<0<\omega$, is finite. Indeed, letting $f_1x^2+f_2x+f_3\in\Zbb[x]$ be the primitive polynomial of the pair, we must have $0>\omega\omega'=(f_2^2-D)/(4f_1^2)$. Therefore we have the bound $\abs{f_2}<\sqrt{D}$ and, since $f_1f_3=(f_2^2-D)/4$, there are finitely many possibilities for $(f_1,f_2,f_3)$, and thus for $(\omega,\omega')$. Remembering that $H_i\subseteq I_i[0,\infty]$ and $K_i\subseteq I_i [\infty,0]$, a fortiori the set $Q_i(D)$ of all pairs $(\omega,\omega')\in H_i\times K_i$ such that $\omega$ and $\omega'$ are conjugate of discriminant $D$ is finite for every~$i$.

Assume now (iii), let $D$ be the discriminant of $\omega_0$, and let
\[
\Qcal_i(D)=\bigl\{(\abf,\bbf):\abf\in E_i^\omega,\bbf\in(\du{E})_i^\omega,\bigl(\pi(\abf),\pi(\bbf)\bigr)\in Q_i(D)\bigr\}.
\]
The map $\Scal:\bigcup_i\Qcal_i(D)\to\bigcup_i\Qcal_i(D)$ given by $\Scal(\abf,\bbf)=(S\abf,a_0\bbf)$ is well defined and is a bijection. Indeed, say $(\abf,\bbf)\in\Qcal_i(D)$;
according to our conventions in Remark~\ref{ref30},
we have $a_0=i\sigma j$ and
$b_0=h\tau i$ for some $\sigma,\tau\in\Sigma$ and nodes $j,h$. Thus $S\abf\in E_j^\omega$ and 
\[
a_0\bbf=
(i\sigma j)(h\tau i)b_1\ldots\in(\du{E})_j^\omega.
\]
Since $\pi(S\abf)=B_{a_0}\m\bigl(\pi(\abf)\bigr)$ and $(a_0\bbf)=D_{a_0}\bigl(\pi(\bbf)\bigr)=B_{a_0}\m\bigl(\pi(\bbf)\bigr)$, the relation of being Galois conjugate and the discriminant are preserved; therefore $\Scal(\abf,\bbf)\in\Qcal_j(D)$. Clearly $\Scal$ is invertible, with $\Scal\m(\abf,\bbf)=(b_0\abf,S\bbf)$.

Lift the pair $(\omega_0,\omega_0')\in H_{i_0}\times K_{i_0}$ in (iii) to some $(\abf,\bbf)\in\Qcal_{i_0}(D)$. We proved above that $\bigcup_i Q_i(D)$ is finite, and Theorem~\ref{ref17}(1) yields that $\bigcup_i \Qcal_i(D)$ is finite, too. As $\Scal$ is a bijection, there exists $p\ge1$ such that $\Scal^p(\abf,\bbf)=(\abf,\bbf)$; therefore $\abf=(a_0\ldots a_{p-1})^\omega$ and $\bbf=(a_{p-1}\ldots a_0)^\omega$. This shows (i) and, by Theorem~\ref{ref17}(2), proves that $(\abf,\bbf)$ is the unique lift of $(\omega,\omega')$; thus, \emph{a posteriori}, $Q_i(D)$ and $\Qcal_i(D)$ are in bijection for every~$i$.

We have thus proved the equivalence of (i), (ii), (iii). The two final sentences in our statement are now automatic from the construction.
\end{proof}

\section{Jump transformations and purely periodic points}\label{ref26}

In this last section we discuss how passing from a Gauss-type map $F$ to a jump acceleration $F_{\jump}$ affects the set of purely periodic points. Let $\set{I_i}$ be a realization of the abstract continued fraction $\Acal$. By Theorem~\ref{ref28},
the attractor $\bigcup_iH_i$ contains intervals, and hence a rational point which, possibly after a global conjugation, we assume to be $0$.
For each node $i$, let
\[
\Pcal_i=\set{a\in E_{ii}:\text{$a$ is parabolic and $B_a(0)=0$}}.
\]
Every $\Pcal_i$ is either empty or a singleton (if $a,b\in\Pcal_i$ then the products $ab$ and $ba$ are both defined in $\widetilde\Sigma$ and equal; since $\Acal$ is a code, we must have $a=b$). We assume that $\Pcal=\bigcup_i\Pcal_i$ is not empty, and set $\Jcal=\Acal\setminus\Pcal\not=\emptyset$.

\begin{lemma}
Let\label{ref22} $\abf\in E^\omega$; then no distinct letters $a,b$, both of them in $\Pcal$, may appear consecutively in~$\abf$. Moreover, either $\abf$ contains infinitely many letters in $\Jcal$, or it ends with $a^\omega$, for some $a\in\Pcal$. In this second case, $\pi(\abf)\in\Qbb$.
\end{lemma}
\begin{proof}
The first two statements are clear, since if $i\not=j$ then no element of $E_{jj}$ can follow an element of $E_{ii}$.
For the last statement, assume $\abf=a_0\ldots a_{r-1}a^\omega$, for some $a\in\Pcal$. Then
\[
\pi(\abf)=B_{a_0\ldots a_{r-1}}
\bigl(\pi(a^\omega)\bigr)=B_{a_0\ldots a_{r-1}}(0)\in\Qbb.
\]
\end{proof}

The multivalued map $F_{\jump}$ from $(\bigcup_iH_i)\setminus\Qbb$ to itself is then defined by setting $F_{\jump}(x)=\pi(S^{e(\abf)+1}\abf)$, where $\pi(\abf)=x$ and $e(\abf)=\min\set{t\ge0:a_t\in\Jcal}$.
Thus $F_{\jump}$ is an acceleration of $F$, and
Theorem~\ref{ref17} implies that the number of $F_{\jump}$-orbits of $x$ is finite and uniformly bounded, and that $F_{\jump}$ is single-valued along purely periodic orbits.

For every $i$, let
\begin{equation}
\begin{split}\label{eq6}
R_i&=\pi\bigl[\set{\abf\in E_i^\omega:a_0\in\Jcal}\bigr]\\
&=\bigcup\bigl\{D_b[K_j]:\text{($j\not=i$ and $b\in E_{ji}$) or ($j=i$ and $b\in E_{ii}\setminus\Pcal_i$)}\bigr\}.
\end{split}
\end{equation}
If $\Pcal_i$ is empty then $K_i=R_i$.
Otherwise, if $\Pcal_i=\set{a}$, then
\begin{equation}\label{eq12}
K_i=\bigcup_{t\ge0}D_a^t[R_i]\cup\set{0};
\end{equation}
this follows by recursively nesting the identity
$K_i=R_i\cup D_a[K_i]$
and observing that $\bigcap_{t\ge0}D_a^t[K_i]=\set{0}$.

\begin{remark}
The\label{ref33} above construction clarifies the structure of the attractor in Lemma~\ref{ref35} and in all our subsequent examples. The key difficulty in determining $\set{K_i}$ is guessing the ``basic blocks'' $\set{R_i}$; once this is done, then each $K_i$ is the closure (which amounts to the final adding of $0$) of the $D_a$-orbit of $R_i$, where $a$ is the parabolic letter in $\Pcal_i$. Lemma~\ref{ref35} is nothing else than the proof that the choice $R_0=R_1=[\infty,-2]$ is the correct one for the $(\tau-1)$-cf of~\S\ref{ref13}.
\end{remark}

Let $\du{F}_{R}:R\to R$ be the first-return map induced by $\du{F}$ on $R=(\bigcup_iR_i)\setminus\Qbb$, namely $\du{F}_R(x)=(\du{F})^{q(x)}(x)$, where $q(x)=\min\set{t\ge1:(\du{F})^t(x)\in R}$.

\begin{theorem}
Assume\label{ref31} the hypotheses of Theorem~\ref{ref19}, and 
further assume that $H_i\cap H_j\subseteq\Qbb$ for $i\not=j$.
Let $F_{\jump}$ be defined as above. 
Given $\omega_0\in H_{i_0}\setminus\Qbb$ the following statements are equivalent.
\begin{itemize}
\item[(i)] At least one of the $F_{\jump}$-orbits of $\omega_0$ is purely periodic.
\item[(ii)] $\omega_0$ has precisely one $F_{\jump}$-orbit, and that orbit is purely periodic.
\item[(iii)] $\omega_0$ is a quadratic irrational and $\omega_0'\in R_{i_0}$.
\end{itemize}
If this happens and the realization is geometric, then $\omega_0'$ has a unique $\du{F}_R$-orbit, which is purely periodic of the same period $p$ as the $F_{\jump}$-orbit of $\omega_0$, and the identity $(\du{F}_R)^t(\omega'_0)=\bigl(F_{\jump}^{-t}(\omega_0)\bigr)'$ holds for every $t\pmod{p}$.
\end{theorem}
\begin{proof}
This is best seen in terms of symbolic dynamics. Consider the infinite alphabet $\Wcal$ whose elements are all words $w=a^eb$, with $a\in\Pcal_i$, $b\in E_{ij}\cap\Jcal$, and $e\ge0$; words in $\Wcal$ can be concatenated if and only if the corresponding paths in $\Gcal_\Acal$ can be concatenated. The key observation here is that infinite sequences $\wbf$ over $\Wcal$ are in $1$--$1$ correspondence with infinite paths in $\Gcal_\Acal$ containing infinitely many letters in $\Jcal$.

By definition, every $F_{\jump}$-orbit of $\omega_0$ is of the form $\pi(S^t\wbf)$, with $S$ denoting word shift. Assume we have two such orbits and that the first one is purely periodic. Then these are induced as above by two sequences $\wbf$ and $\wbf'$, with $\wbf$ purely periodic. By Theorem~\ref{ref17}(2) $\wbf$ and $\wbf'$ are equal as elements of $E_{i_0}^\omega$ and hence, by the observation above, are equal as sequences over $\Wcal$; this shows that (i) and (ii) are equivalent.

Assume (ii) and let $\wbf=(w_0\ldots w_{p-1})^\omega$ be such that $F_{\jump}^t(\omega_0)=\pi(S^t\wbf)$ for every $t$. Then clearly $\omega_0$ is a quadratic irrational and $w_0\ldots w_{p-1}$, thought as a word in $\Acal$, is the symbolic period of $(i_0,\omega_0)$ under $\Fcal$. By Theorem~\ref{ref19} this symbolic period is unique and is the reverse of the symbolic period of $(i_0,\omega_0')$ under $\du{\Fcal}$. Let $b\in\Acal$ be the last letter in $w_{p-1}$. Then $b\in E_{ji_0}\cap\Jcal$ for some $j$, and therefore $\omega_0'\in D_b[K_j]\subseteq R_{i_0}$ by~\eqref{eq6}; thus (iii) holds.

The argument can be reversed: assuming (iii), the unique symbolic period of $(i_0,\omega_0')$ under $\du{\Fcal}$ must begin with a letter in $\Jcal$. Thus, the reverse of this symbolic period is of the form $w_0\ldots w_{p-1}$ for certain words $w_t\in\Wcal$, and (i) holds.

We now prove our last statement. Let $w_0\ldots w_{p-1}$ be the word over $\Wcal$ that, as a word over $\Acal$, is the symbolic period of $(i_0,\omega_0)$ under $\Fcal$; say that $w_t=a_t^{e_t}b_t$. By Theorem~\ref{ref19}, the symbolic period of $(i_0,\omega'_0)$ under $\du{\Fcal}$ is $b_{p-1}a_{p-1}^{e_{p-1}}\ldots b_0a_0^{e_0}$. Working by induction on $p$, it suffices to show that $(\du{F})^{e_{p-1}+1}(\omega'_0)\in R$,
that $(\du{F})^q(\omega'_0)\notin R$ for every $1\le q\le e_{p-1}$, and that $(\du{F}_R)(\omega'_0)=\bigl(F_{\jump}\m(\omega_0)\bigr)'$.
The first statement is true because $(\du{F})^{e_{p-1}+1}(\omega'_0)$ is the $\pi$-image of an infinite path in $\Gcal_{\du{\Acal}}$ that begins with $b_{p-2}\in\Jcal$ (index modulo $p$). If the second statement were false for some $q$, we would have that $(\du{F})^q(\omega'_0)$ is, on the one hand, the $\pi$-image of
\[
\bigl(a_{p-1}^{e_{p-1}-q+1}b_{p-2}a_{p-2}^{e_{p-2}}\ldots
b_0a_0^{e_0}b_{p-1}a_{p-1}^{q-1}\bigr)^\omega,
\]
and on the other the $\pi$-image of an infinite path in $\Gcal_{\du{\Acal}}$ that begins with a letter in $\Jcal$;
since $e_{p-1}-q+1\ge1$ and $a_{p-1}\notin\Jcal$,
this contradicts Theorem~\ref{ref17}(2). Finally, the third one follows from the corresponding statement in Theorem~\ref{ref19}, as
\[
(\du{F}_R)(\omega'_0)=
(\du{F})^{e_{p-1}+1}(\omega'_0)=
\bigl(F^{-(e_{p-1}+1)}(\omega_0)\bigr)'=
\bigl(F_{\jump}\m(\omega_0)\bigr)'.
\]
\end{proof}

We conclude this paper by applying Theorems~\ref{ref19} and \ref{ref22} to the cases in Corollary~\ref{ref4}(2)--(5).

\begin{example}
The\label{ref38} Ceiling map is the jump transformation determined by taking $a=0l0$, $b=0n0$, $\Acal=\set{a,b}$, $I_0=[0,1]$, $\Pcal_0=\set{a}$. Since $\Acal=\du{\Acal}$, we plainly have $H_0=[0,1]$ and $K_0=[1,0]$. Thus $R=D_b[1,0]=[1,\infty]$ and we obtain Corollary~\ref{ref4}(2). Explicit computation shows that the dual $\du{F}_{R}$ of the Ceiling map is conjugate via $x\mapsto -x\m$ to the R\'enyi transformation~\cite{renyi57b} $x\mapsto -x\m-\ceil{-x\m}$ on $[-1,0]$.
\end{example}

\begin{example}
Treating\label{ref23} the Even map is similar, and actually simpler, to treating the $(\tau-1)$-cf of \S\ref{ref13}. The slow version is given by taking $n=2$, $\Acal=\set{0n^20,1l1,0l0,0nl1,1nf0,1lnf1}$, $I_0=H_0=[-1,0]$, $I_1=H_1=[0,1]$. We again denote the elements of $\Acal$ by $o,p,q,r,s,t$, in this order; similarly to the case of \S\ref{ref13}, they match in pair, with $B_s=B_t$ and $B_q=B_r$ as matrices. We have $\Pcal_0=\set{o}$, $\Pcal_1=\set{p}$, and the resulting $F_{\jump}$ is precisely the Even map of  Figure~\ref{fig9} left.
\begin{figure}[ht]
\includegraphics[width=5cm]{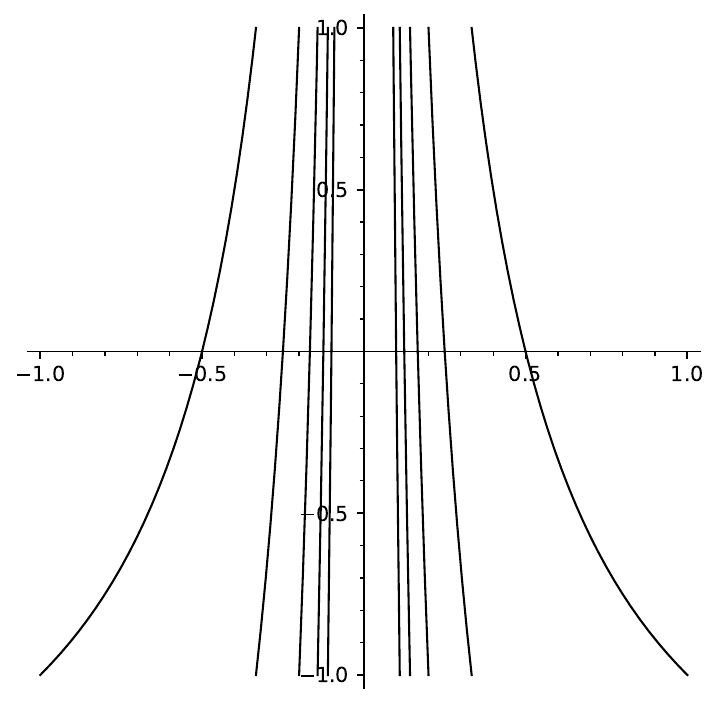}
\hspace{0.6cm}
\includegraphics[width=5cm]{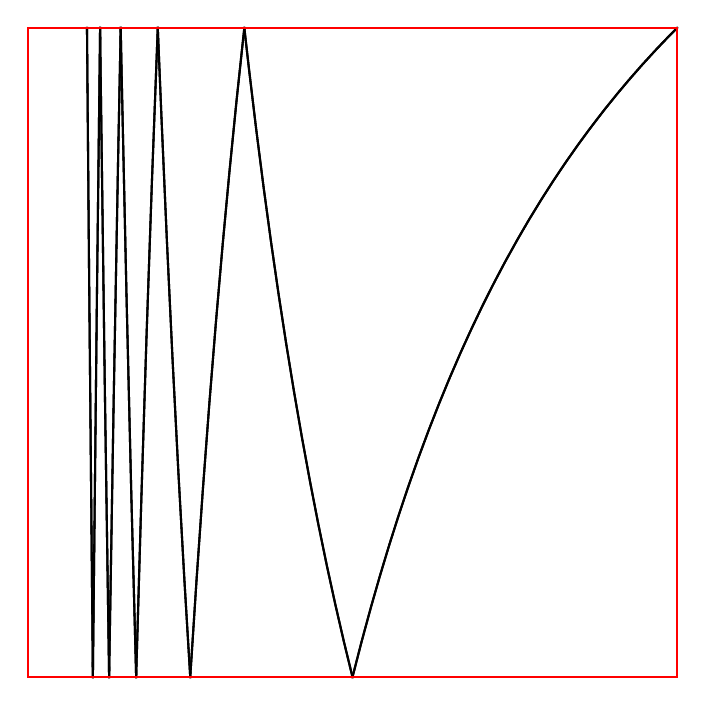}
\caption{The unfolded Even map and its dual, the folded one}
\label{fig9}
\end{figure}

The attractor $(K_0,K_1)$ arises as in Lemma~\ref{ref35} by taking $R_0=R_1=[\infty,-1]$. All the discussion in~\S\ref{ref13} for the $(\tau-1)$-cf carries through with the appropriate modifications. We plot in Figure~\ref{fig9} right the map $\du{F}_R$,
which is conjugate via $x\mapsto -x\m$ to the folded version on $[0,1]$ of Even,
\[
x\mapsto\bigl\vert x\m-\text{(the even integer nearest to $x\m$)}\bigr\vert.
\]
\end{example}

\begin{example}
The\label{ref32} Odd and Nearest Integer cases are similar, but more involved than the previous ones; we treat the Odd case in some detail and sketch the modifications needed for Nearest Integer. The graph $\Gcal_\Acal$ and the slow Odd map are shown in Figure~\ref{fig10}; the names of the edges mirror those in Figure~\ref{fig1} and we have $I_0=H_0=[-1,0]$, $I_1=H_1=[0,1]$.
Differently from the cases in~\S\ref{ref13} and in Example~\ref{ref23}, here $R_0\not=R_1$; indeed, setting $R_0=[\infty,-\tau-1]$, $R_1=[\infty,-\tau+1]$ the statement and the proof of Lemma~\ref{ref35} carry through. The first part of Theorem~\ref{ref31} now yields Corollary~\ref{ref4}(4).

A complication arises with the dual map; Figure~\ref{fig11} is the Odd analogue of Figure~\ref{fig4} for the $(\tau-1)$-cf. The branches $D_q\m$ and $D_r\m$ do not collapse anymore, and neither do $D_s\m$ and $D_t\m$. The immediate consequence is that the realization $(I_0,I_1)$ is not geometric.
Consider for example any point $\alpha$ in the domain $[-3,-\tau-1]$ of the second from right incomplete branch in Figure~\ref{fig11} right. As a point of $K_0$, only $D_q\m$ may act on $\alpha$, while as a point of $K_1$ only $D_r\m$ may act. The two images are different, so
$\du{\Fcal}$ does not descend to a Gauss-type map. 
However, due to the identities $D_t\m=D_p\m D_s\m$ and $D_q\m=D_o\m D_r\m$ (that can be checked directly, or more interestingly by realizing them as identities on the $B_a\m$-branches in Figure~\ref{fig10} right), the first-return map $\du{\Fcal}_{R_0\cupdot R_1}$ does descend to a map $\du{F}_R$ on $R=(R_0\cup R_1)\setminus\Qbb(\tau)$, which is shown in Figure~\ref{fig11} right.

We stress an issue related to our discussion in Remark~\ref{ref34}: $\du{F}_R$ is an excellent dual to Odd in the measure-theoretic sense. However, since the realization is not geometric, it is \emph{not} a dual in the sense of the second part of Theorem~\ref{ref31}. For example, the point $\omega=\tau-2$ is fixed under Odd, with symbolic period $q$. However, $\omega'=-\tau-1$ has two $\du{F}_R$-images, namely itself and $-\tau+1$, and uncountably many $\du{F}_R$-orbits.

All of the above holds for the Nearest Integer case with 
the appropriate modifications:
$\Acal$ equals now $\set{0n0, 1l1, 0ln0, 0l^21, 1nlf0, 1n^2f1}$, $I_0=H_0=[-1/2,0]$, and $I_1=H_1=[0,1/2]$. Taking $R_0=[\infty,-\tau-1]$ and $R_1=[\infty,-\tau]$,
Lemma~\ref{ref35} carries through, and we obtain Corollary~\ref{ref4}(5). The realization $(I_0,I_1)$ is not geometric, and considerations as above hold for $\du{F}_R$.

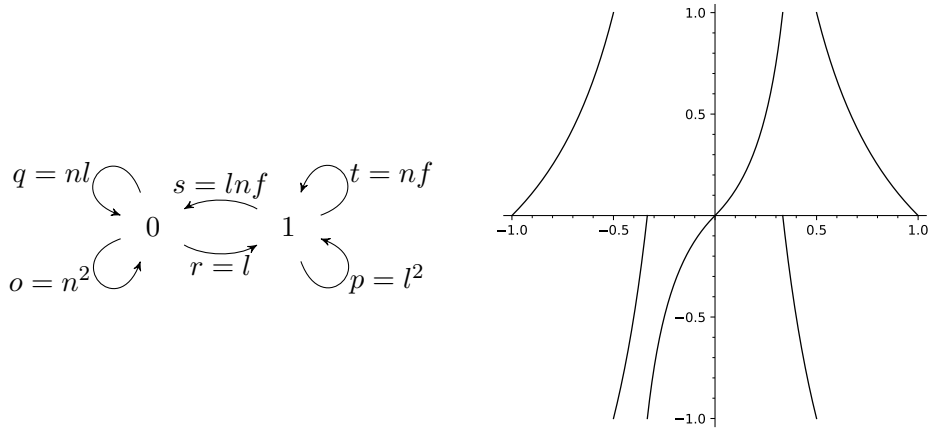
\begin{figure}[ht]
\raisebox{1.5cm}{
\begin{tikzpicture}[scale=0.9]
\node (0) at (0,0)  []  {$0$};
\node (1)  at (2,0) []  {$1$};
\path
(0) edge[Nedge,out=110,in=160,loop] node[left] {$q=nl$}
(0) edge[Nedge,out=-160,in=-110,looseness=8,loop] node[left] {$o=n^2$}
(0) edge[Nedge,out=-30,in=-150] node[below,inner sep=1pt] {$r=l$} (1)
(1) edge[Nedge,out=-70,in=-20,loop] node[right] {$p=l^2$}
(1) edge[Nedge,out=20,in=70,looseness=8,loop] node[right] {$t=nf$}
(1) edge[Nedge,out=150,in=30] node[above,inner sep=1pt] {$s=lnf$} (0);
\end{tikzpicture}
}
\quad
\includegraphics[width=6cm]{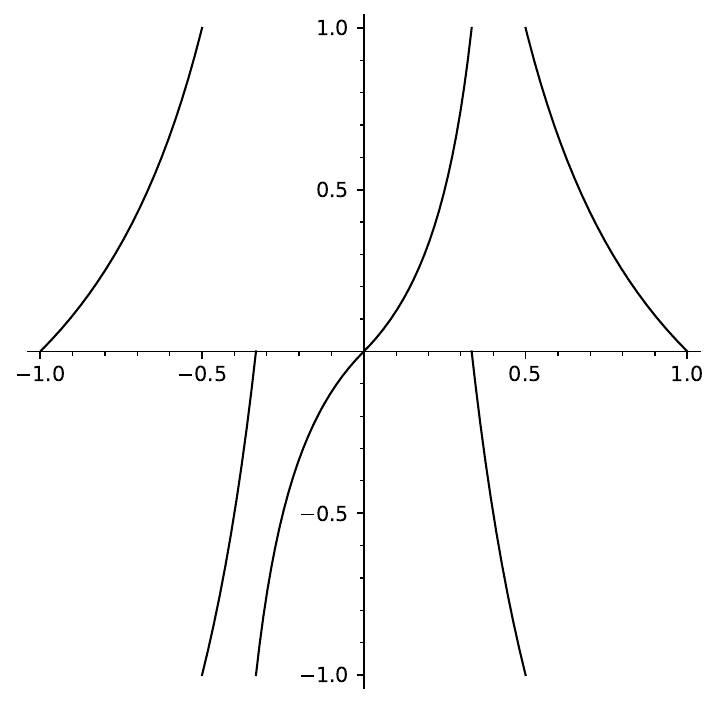}
\caption{The graph $\Gcal_\Acal$ and the slow Odd map.
The branches, listed in the order they touch the $x$-axis, are $B_r\m,B_q\m,B_o\m,B_p\m,B_s\m,B_t\m$}
\label{fig10}
\end{figure}

\begin{figure}[ht]
\includegraphics[width=5.5cm]{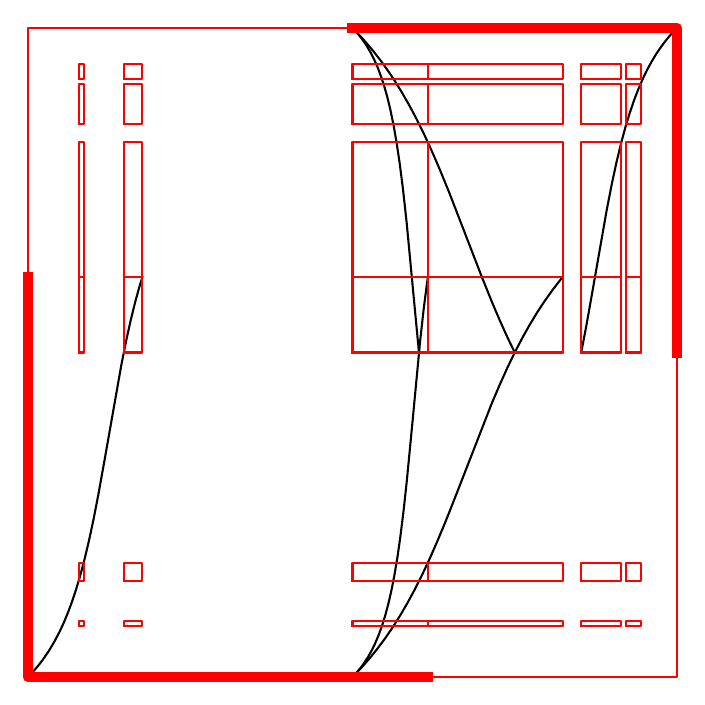}
\hspace{0.6cm}
\includegraphics[width=5.5cm]{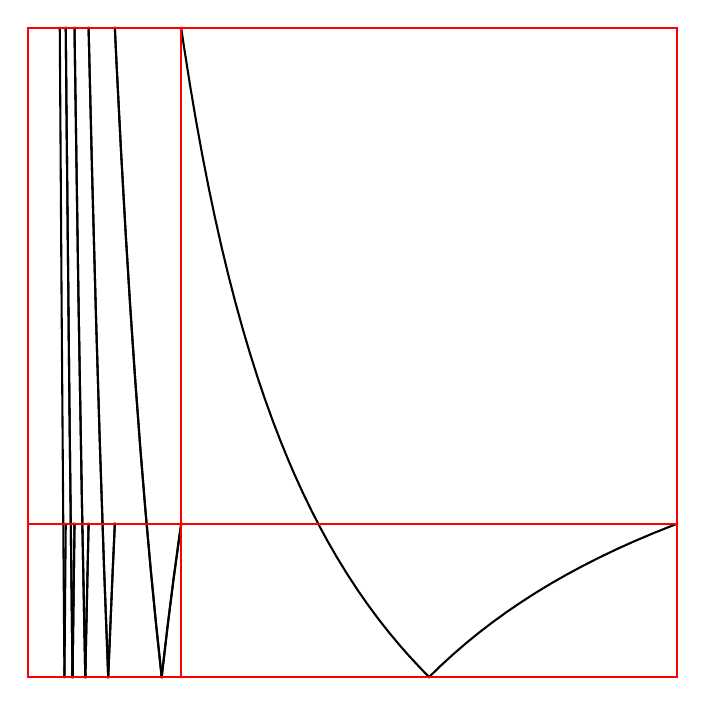}
\caption{Left: the analogue of Figure~\ref{fig4} for the slow Odd map. The increasing branches in the left picture are 
$D_o\m,D_q\m,D_r\m,D_p\m$, and the decreasing ones 
$D_s\m,D_t\m$, in both cases listed from left to right.
Right: the first-return map $\du{F}_R$, grid at $\set{\infty,-\tau-1,-\tau+1}$}
\label{fig11}
\end{figure}
\end{example}

\bibliography{bibliography}

\end{document}